\documentclass[preprint,a4paper,12pt]{elsarticle}

\usepackage{amsmath}
\usepackage{latexsym}
\usepackage{amssymb}
\usepackage{pstricks}
\usepackage{mathrsfs}
\usepackage{color}
\usepackage{amsthm}

\newtheorem{thm}{Theorem}[section]
\newtheorem{lem}[thm]{Lemma}

\theoremstyle{definition}
\newtheorem{dfn}{Definition}
\theoremstyle{remark}
\newtheorem*{rmk}{Remark}
\newdefinition{crl}{Corollary}

\newcommand{\de}{\mathrm{d}}

\newcommand{\bfb}{\mbox{\boldmath{$b$}}}
\newcommand{\bfx}{\mbox{\boldmath{$x$}}}
\newcommand{\bfn}{\mbox{\boldmath{$n$}}}
\newcommand{\bff}{\mbox{\boldmath{$f$}}}
\newcommand{\bfF}{\mbox{\boldmath{$F$}}}
\newcommand{\bfg}{\mbox{\boldmath{$g$}}}

\newcommand{\bfu}{\mbox{\boldmath{$u$}}}
\newcommand{\bfw}{\mbox{\boldmath{$w$}}}
\newcommand{\bfv}{\mbox{\boldmath{$v$}}}

\newcommand{\bfa}{\mbox{\boldmath{$a$}}}

\newcommand{\bfG}{\mbox{\boldmath{$G$}}}

\newcommand{\bfy}{\mbox{\boldmath{$y$}}}

\newcommand{\bfz}{\mbox{\boldmath{$z$}}}

\newcommand{\bfr}{\mbox{\boldmath{$r$}}}
\newcommand{\bfj}{\mbox{\boldmath{$j$}}}
\newcommand{\bfs}{\mbox{\boldmath{$s$}}}

\newcommand{\bfA}{\mbox{\boldmath{$A$}}}

\newcommand{\bfB}{\mbox{\boldmath{$B$}}}

\newcommand{\bfmu}{\mbox{\boldmath{$\mu$}}}

\newcommand{\bfphi}{\mbox{\boldmath{$\phi$}}}

\newcommand{\bfsigma}{\mbox{\boldmath{$\sigma$}}}

\newcommand{\bfvarphi}{\mbox{\boldmath{$\varphi$}}}

\newcommand{\bfxi}{\mbox{\boldmath{$\xi$}}}

\newcommand{\JZtext}[1]{#1}

\begin{document}

\begin{frontmatter}



\title{Some Properties of Strong Solutions
to Nonlinear Heat and Moisture Transport in Multi-layer Porous
Structures}


\author[CIDEAS,KM]{\corref{cor1}Michal Bene\v{s}}
\cortext[cor1]{Corresponding address: Department of Mathematics,
Faculty of Civil Engineering, Czech Technical University in Prague,
Th\'{a}kurova 7, 166 29 Prague 6, Czech Republic}
\ead{xbenesm3@fsv.cvut.cz}

\author[FCEKSM]{Jan Zeman}

\address[CIDEAS]{Centre for Integrated Design of Advanced Structures,}
\address[KM]{Department of Mathematics,}
\address[FCEKSM]{Department of Mechanics,\\Faculty of Civil Engineering,\\
Czech Technical University in Prague,\\
Th\'{a}kurova 7, 166 29 Prague 6, Czech Republic}

\begin{abstract}

The present paper deals with mathematical models of heat and
moisture transport in layered building envelopes. The study of such
processes generates a system of two doubly nonlinear evolution
partial differential equations with appropriate initial and boundary
conditions. The existence of the strong solution in two dimensions
on a (short) time interval is proven. The proof rests on regularity
results for elliptic transmission problem for isotropic composite
materials.

\end{abstract}

\begin{keyword}

Initial-boundary value problems for second-order parabolic systems
\sep local existence, smoothness and regularity of solutions \sep
coupled heat and mass transport


\MSC 35K51 \sep 35A01 \sep 35B65

\end{keyword}

\end{frontmatter}

%\section{}
%\label{}

%--------------------------------------------------------------------------

\section{Introduction}
Building envelopes, which act as barriers between the indoor and
outdoor environments, present a crucial component responsible for
the building's performance over the whole service life. In this
regard, an important requirement to achieve an energy-efficient
design is the assessment of the heat and moisture behavior of the
component when exposed to natural climatic conditions. This task can
hardly be accomplished by purely experimental means, mainly due to
the long-term character of environmental variations and the
transport processes involved. Therefore, a considerable research
effort has been devoted to the development of predictive models for
coupled heat and mass transfer in building materials, see
e.g.~\cite{CernyRovnanikova2002,Hens:2007} for historical overviews.

The major challenge in predicting the transport phenomena in
building components lies in their complex porous microstructure,
resulting in an intricate mechanism of moisture absorption from
surrounding environment. Here, the dominant physical processes
involve adsorption forces, attracting vapor phase molecules to solid
parts of the porous system, and capillary condensation in pores.
This needs to be complemented with non-linear dependence of thermal
conduction on temperature and water content. As a result,
engineering models of simultaneous heat and moisture transfer are
posed in the form of strongly coupled parabolic system with highly
non-linear coefficients. Discretization of these equations,
typically based on finite volume or finite element methods, then
provides the basis for numerous simulation tools used in engineering
practice, see e.g.~\cite{Kalagasidis:2007:IBPT} for a recent survey.
However, to our best knowledge, the qualitative properties of the
resulting systems remain largely unexplored.

The mathematical models of transport processes in porous composite
media consist of the balance equations, governing the conservation
of mass (moisture) and thermal energy, supplemented by the
appropriate boundary, transmission and initial conditions. This
system can be written in the form
\begin{equation}\label{eq01}
\frac{\partial B^j_{\ell}(\bfu_{\ell})}{\partial t} - \nabla \cdot
\bfA^j_{\ell}(\bfu_{\ell},\nabla\bfu_{\ell}) =
f^j_{\ell}(\bfu_{\ell}) \quad {\rm in}\, Q_{\ell T}, \quad j=1,2,
\end{equation}
with the nonlinear boundary conditions
\begin{equation}\label{eq02}
-\bfA^j_{\ell}(\bfu_{\ell},\nabla\bfu_{\ell})\cdot
\bfn_{\ell}(\bfx)=g^j_{\ell}(\bfx,t,u^j_{\ell}) \quad {\rm on}\,
S_{\ell T},   \quad j=1,2,
\end{equation}
the so-called transmission conditions
\begin{equation}\label{transmission_conditions_01}
\left\{
\begin{array}{rclll}
u^j_{\ell} &=& u^j_{m} & {\rm on }\; \Sigma_{m\ell}^T,&
\\
-\bfA^j_{\ell}(\bfu_{\ell},\nabla\bfu_{\ell})\cdot\bfn_{\ell}(\bfx)
&=&-\bfA^j_{m}(\bfu_{m},\nabla\bfu_{m})\cdot\bfn_{\ell}(\bfx) & {\rm
on}\;  \Sigma_{m\ell}^T,&
\end{array}\right.
\end{equation}
$j=1,2$, and the initial condition
\begin{equation}\label{eq03}
\bfu_{\ell}(\bfx,0)=\bfmu_{\ell}(\bfx) \quad {\rm in} \,
\Omega_{\ell} .
\end{equation}
Here, $\Omega$ represents a two-dimensional bounded domain with a
Lipschitz-continuous boundary $\Gamma=\partial\Omega$;
$\bfn=(n_1,n_2)$ denotes the outer unit normal to $\Gamma$. $\Omega$
consists of $M$ disjoint subdomains $\Omega_{\ell}$ with
boundary $\partial \Omega_{\ell}$, $\ell=1,\dots,M$, separated by
smooth internal interfaces
$\Gamma_{m\ell}=\partial\Omega_m\cap\partial\Omega_{\ell}\neq
\emptyset$. For a fixed positive $T$, we denote by $Q_{\ell T}$ the
space-time cylinder $Q_{\ell T}=\Omega_{\ell}\times(0,T)$, similarly
{ $S_{\ell T}=(\partial\Omega_{\ell}\cap\Gamma)\times (0,T) $} and
$\Sigma_{m\ell}^T=\Gamma_{m\ell}\times(0,T)$. Further, in
\eqref{eq01}--\eqref{eq03}, $\bfu_{\ell}=(u^1_{\ell},u^2_{\ell})$
represents the unknown fields of state variables and the vector
$\bfmu_{\ell}=({\mu}^1_{\ell},{\mu}^2_{\ell})$ describes the initial
condition. By $\bfB_{\ell}$, $\bfA^j_{\ell}$, $\bff_{\ell}$,
$\bfg_{\ell}$, we denote the vectors
$\bfB_{\ell}=(B^1_{\ell},B^2_{\ell})$,
$\bfA^j_{\ell}=(A^{j1}_{\ell},A^{j2}_{\ell})$,
$\bff_{\ell}=(f^1_{\ell},f^2_{\ell})$,
$\bfg_{\ell}=(g^1_{\ell},g^2_{\ell})$, which are functions of
primary unknowns $\bfu_{\ell}$. Hence, the problem is strongly
nonlinear.

The existence of weak solutions to the system~\eqref{eq01} in
homogeneous bounded domains ($\ell=1$) subject to mixed boundary
conditions with homogeneous Neumann boundary conditions has been
shown by Alt and Luckhaus in~\cite{AltLuckhaus1983}. They obtained
an existence result assuming the operator $\bfB$ in the parabolic
part to be only (weak) monotone and subgradient.
This result has been extended in various different directions. Filo
and Ka\v{c}ur~\cite{FiloKacur1995} proved the local existence of the
weak solution for the system with nonlinear Neumann boundary
conditions and under more general growth conditions on
nonlinearities in $\bfu$. These results are not applicable if $\bfB$
does not take the subgradient form, which is typical of coupled heat
and mass transport models.

In this context, the only related works we are aware of are due to
Vala~\cite{Vala2002}, Li and Sun~\cite{LiSun2010} and Li~\emph{et
al.}~\cite{LiSunWang2010}.
Nonetheless, even though~\cite{Vala2002} admits non-symmetry in the
parabolic part, it requires unrealistic symmetry in the elliptic
term. The latter works, studying a model arising from textile
industry, prove the global existence for one-dimensional problem
using the Leray-Schauder fixed point theorem. The proofs, however,
exploit the specific structure of the model and as such are not
applicable to our general setting.

In this paper we adapt ideas presented by Giaquinta and Modica
in~\cite{GiaquintaModica1987} and Weidemaier in
\cite{Weidemaier1991}, where the local solvability of quasilinear
diagonal parabolic systems is proved, to show the local existence of
strong solution to the general transmission
problem~\eqref{eq01}--\eqref{eq03} for isotropic media under less
restrictive assumptions on the operator $\bfB(\bfu)$ and the
parabolicity condition of the problem. The main result (local in
time existence) is proved by means of a fixed point argument based
on the Banach contraction principle.

The paper is organized as follows. In Section \ref{Preliminaries},
we introduce the appropriate function spaces and recall important
embeddings and interpolation-like inequalities needed below together
with some auxiliary results. In Section
\ref{sec_structure_conditions}, we specify our assumptions on data
and structure conditions and introduce the precise definition of
admissible domains describing the composite body under which the
main result of the paper is proved. In Section~\ref{linearproblem},
we prove the existence and uniqueness of the solution to an
auxiliary linearized problem using the regularity result for
elliptic systems in composite-like domains. To make the text more
readable, technical details of the proof are collected in
Appendices~\ref{proof_1},~\ref{proof_2}
and~\ref{regularity_stationary}. The main result is proved in
Section~\ref{NonlinearProblem} via the Banach contraction principle.
Finally, in Section~\ref{sec:applications} we present applications
of the theory to selected engineering models of heat and mass
transfer.

%--------------------------------------------------------------------------------

\section{Preliminaries}\label{Preliminaries}
\subsection{Definition of some function spaces and notation}

We denote by $\mathbf{W}_{\ell}^{l,p}\equiv
W^{l,p}(\Omega_{\ell})^2$, $l\ge 0$ ($l$ need not to be an integer)
and $1 \leq p \leq \infty$, the usual Sobolev space of functions
defined in $\Omega_{\ell}$ and by
$\mathbf{W}_{\ell,\Gamma}^{l-1/p,p}\equiv
W^{l-1/p,p}(\partial\Omega_{\ell})^2$ the space of traces of
functions from $\mathbf{W}^{l,p}_{\ell}$ on $\partial\Omega_{\ell}$.
We set $\mathbf{L}^p_{\ell}\equiv\mathbf{W}^{0,p}_{\ell}$.
Let $\mathcal{B}$ be an arbitrary Banach space, then $(\mathcal{B})^*$ represents its dual.
$\bfphi'(t)$ indicates the partial derivative with respect to time;
we also write
\begin{displaymath}
\bfphi'(t) := \frac{\partial \bfphi}{\partial t}  .
\end{displaymath}
In order to define the concept of strong solution, we will make use
of the following Banach spaces
\begin{multline*}
\mathcal{X}_{\ell,T} := \left\{ \bfphi; \,  { \bfphi'(t)} \in
L^{\infty}(0,T;\mathbf{L}^{2}_{\ell}), \, \bfphi'(t) \in
L^2(0,T;\mathbf{W}^{2,2}_{\ell}), \right.
\\
\left.  \bfphi''(t) \in L^2(0,T;\mathbf{L}^2_{\ell}),\;
\bfphi(0)={\bf0} \right\}
\end{multline*}
and
\begin{displaymath}
\mathcal{Y}_{\ell,T}:=\left\{ \bfvarphi; \,  \bfvarphi'(t) \in
L^2(0,T;\mathbf{L}^2_{\ell}),  \,
\bfvarphi(0)\in\mathbf{W}^{1,2}_{\ell}\right\},
\end{displaymath}
respectively, equipped with the norms
\begin{equation}
\|\bfphi\|_{\mathcal{X}_{\ell,T}}  := \|{ \bfphi'(t)}
\|_{L^{\infty}(0,T;\mathbf{L}^{2}_{\ell})} +
\|\bfphi'(t)\|_{L^2(0,T;\mathbf{W}^{2,2}_{\ell})} +
\|\bfphi''(t)\|_{L^2(0,T;\mathbf{L}^{2}_{\ell})}
\end{equation}
and
\begin{equation}
\|\bfvarphi\|_{\mathcal{Y}_{\ell,T}}  :=
\|\bfvarphi'(t)\|_{L^2(0,T;\mathbf{L}^2_{\ell})}
+\|\bfvarphi(0)\|_{\mathbf{W}^{1,2}_{\ell}} ,
\end{equation}
respectively.

Throughout the paper, $\ell$ and $m$ are assumed to always range
from $1$ to $M$ and $m \not= \ell$, whereas indices $i,j=1,2$.
Unless specified otherwise, we use Einstein's summation convention
for indices running from 1 to 2. We shall denote by $c, c_1,
c_2,\dots$ generic constants independent on $T$ having different
values in different places. Let us stress that throughout the paper
the function $C=C(T)$ depends solely on $T$ and $C(T)\rightarrow
0_+$ for $T\rightarrow 0_+$.

\subsection{Some embeddings and interpolation like-inequalities}
In the paper we shall use the following embeddings (recall that
$\Omega$ is a two-dimensional bounded domain)(see
\cite{AdamsFournier1992,KufFucJoh1977}):

\begin{equation}\label{embedding_theorems}
\left\{
\begin{array}{lll}
\mathbf{W}^{1,2}_{\ell}\hookrightarrow \mathbf{L}_{\ell}^p,\;  &
\|\bfphi\|_{\mathbf{L}^p_{\ell}} \leq c \,
\|\phi\|_{\mathbf{W}^{1,2}_{\ell}} & \forall \bfphi \in
\mathbf{W}^{1,2}_{\ell},\; 1\leq p <\infty,
\\
\mathbf{W}^{l,2}_{\ell}\hookrightarrow \mathbf{W}^{1,p}_{\ell}, &
\|\bfphi\|_{\mathbf{W}^{1,p}_{\ell}}\leq
c\,\|\phi\|_{\mathbf{W}^{l,2}_{\ell}} & \forall \bfphi \in
\mathbf{W}^{l,2}_{\ell},\; 1<l<2, p=2/(2-l),
\\
\mathbf{W}^{l,p}_{\ell}\hookrightarrow \mathbf{L}^{\infty}_{\ell}, &
\|\bfphi\|_{\mathbf{L}^{\infty}_{\ell}} \leq c \,
\|\phi\|_{\mathbf{W}^{l,p}_{\ell}} & \forall \bfphi \in
\mathbf{W}^{l,p}_{\ell},\; lp>2.
\end{array}\right.
\end{equation}
Let us present some properties of $\mathcal{X}_{\ell,T}$. Assume
$\bfphi \in \mathcal{X}_{\ell,T}$. Using the interpolation
inequality \cite[Theorem 5.8]{AdamsFournier1992}
\begin{equation}\label{interpolner}
\|\bfphi'(t)\|_{\mathbf{L}^{4}_{\ell}} \leq
c\|\bfphi'(t)\|^{1/4}_{\mathbf{W}^{2,2}_{\ell}}
\|\bfphi'(t)\|^{3/4}_{\mathbf{L}^{2}_{\ell}}
\end{equation}
we obtain
\begin{eqnarray}\label{odh1}
\| \bfphi'(t)\|_{L^{8}(0,T;\mathbf{L}^{4}_{\ell})}&\leq& c
\|\bfphi'(t) \|^{1/4}_{L^2(0,T;\mathbf{W}^{2,2}_{\ell})} \|
\bfphi'(t) \|^{3/4}_{{L}^{\infty}(0,T;\mathbf{L}^{2}_{\ell})}
\nonumber\\
&\leq& c \, \|\bfphi\|_{\mathcal{X}_{\ell,T}}.
\end{eqnarray}
For all $\bfphi \in \mathcal{X}_{\ell,T}$ we have
\begin{eqnarray}\label{odh2}
\|\bfphi\|_{{L}^{\infty}(0,T;\mathbf{L}^{\infty}_{\ell})} \leq c
\|\bfphi\|_{{L}^{\infty}(0,T;\mathbf{W}^{2,2}_{\ell})} &\leq&
cT^{1/2} \|\bfphi'(t)\|_{{L}^{2}(0,T;\mathbf{W}^{2,2}_{\ell})}
\nonumber
\\
&\leq& cT^{1/2} \|\bfphi\|_{\mathcal{X}_{\ell,T}}.
\end{eqnarray}
Further, combining \eqref{embedding_theorems} and the interpolation
inequality \cite[Theorem 5.2]{AdamsFournier1992} we obtain
\begin{equation}\label{interpolner2}
\|\bfphi'(t)\|_{\mathbf{L}^{\infty}_{\ell}} \leq
c\|\bfphi'(t)\|_{\mathbf{W}^{1,3}_{\ell}} \leq
c\|\bfphi'(t)\|_{\mathbf{W}^{4/3,2}_{\ell}} \leq
c\|\bfphi'(t)\|^{2/3}_{\mathbf{W}^{2,2}_{\ell}}
\|\bfphi'(t)\|^{1/3}_{\mathbf{L}^{2}_{\ell}}
\end{equation}
and consequently
\begin{equation}\label{odh3}
\| \bfphi'(t)\|_{L^{3}(0,T;\mathbf{L}^{\infty}_{\ell})} \leq c \|
\bfphi'(t)\|_{L^{3}(0,T;\mathbf{W}^{1,3}_{\ell})}  \leq c \,
\|\bfphi\|_{\mathcal{X}_{\ell,T}}.
\end{equation}

%------------------------------------------------------------------------------------

\section{Structure conditions and admissible
domains}\label{sec_structure_conditions}

In this Section, we summarize our assumptions on the problem data
and specify in detail the geometry of the considered domains.

\subsection{Structure conditions}\label{structure_conditions}
\begin{itemize}
\item[(A1)] { For every $\bfz \in \mathbb{R}^2$,  $B^1_{\ell}(s,z_2)$ and
$B^2_{\ell}(z_1,s)$ are increasing functions (with respect to $s$)},
$B^j_{\ell}:\mathbb{R}^2\rightarrow \mathbb{R}$,
{ such that $|\partial^{\alpha}B^j_{\ell}(\bfz)|$ are bounded on every
bounded set in $\mathbb{R}^2$ for $|\alpha|\leq 3$.}
 Further, we denote
the matrix
\begin{displaymath}
b^{ij}_{\ell}(\bfz) := \frac{\partial B^j_{\ell}(\bfz)}{\partial
z^i};
\end{displaymath}
%Note that (A1) yields $b^{ij}_{\ell}\in W^{2,\infty}(\mathbb{R}^2)$.
%
\item[(A2)] $\bfA^j_{\ell}:\mathbb{R}^2\times\mathbb{R}^{2\times 2}
\rightarrow\mathbb{R}^2$ are continuous and of the semilinear form
\begin{equation}\label{form_A}
\bfA^j_{\ell}(\bfr,\bfs)= \sum_{i=1}^2 a^{ji}_{\ell}(\bfr)\bfs_i,
\end{equation}
for all $\bfr\in \mathbb{R}^2$ and $\bfs_i=(s_i^1,s_i^2)$, where
$s_i^j\in\mathbb{R}^2$ for $i,j=1,2$. Note that in \eqref{form_A}
$\bfr$ stands for $\bfu$ and $\bfs_i$ stands for the vector $\nabla
u^i$. Functions $a^{ji}_{\ell}:\mathbb{R}^2\rightarrow\mathbb{R}$,
are positive, scalar (due to the assumed isotropy of the material)
and { $|\partial^{\alpha}a^{ji}_{\ell}(\bfr)|$ are bounded on every
bounded set in $\mathbb{R}^2$ for $|\alpha|\leq 3$}.
Further we assume
\begin{equation}\label{gen_par1}
b^{11}_{\ell}(\bfmu_{\ell})b^{22}_{\ell}(\bfmu_{\ell})
a^{12}_{\ell}(\bfmu_{\ell})a^{21}_{\ell}(\bfmu_{\ell})>
\left(\frac{b^{12}_{\ell}(\bfmu_{\ell})a^{21}_{\ell}
(\bfmu_{\ell})+b^{21}_{\ell}(\bfmu_{\ell})
a^{12}_{\ell}(\bfmu_{\ell})}{2}\right)^2
\end{equation}
{ in $\overline{\Omega}$ and the ellipticity condition}
\begin{equation}\label{gen_par2}
a^{11}_{\ell}(\bfmu_{\ell})a^{22}_{\ell}(\bfmu_{\ell})>
a^{12}_{\ell}(\bfmu_{\ell})a^{21}_{\ell}(\bfmu_{\ell})
\end{equation}
{ in $\overline{\Omega}$} with $\bfmu_{\ell}$ representing the
initial distribution of the unknown fields $\bfu_{\ell}$;
\item[(A3)]  $\bff_{\ell}:\mathbb{R}^2\rightarrow \mathbb{R}^2$,
 { $|\partial^{\alpha}f^{j}_{\ell}(\bfz)|$ are bounded on every
bounded set in $\mathbb{R}^2$ for $|\alpha|\leq 2$};
\item[(A4)] $\bfg:\Gamma\times (0,T)\times\mathbb{R}^2\rightarrow
\mathbb{R}^2$ is of the form of the Newton-type boundary conditions
\begin{displaymath}
g^j_{\ell}(\bfx,t,\bfu_{\ell})=\alpha^j_{\ell}(u^j_{\ell}-\sigma^j(\bfx,t)),
\end{displaymath}
where $\alpha^j_{\ell}$ are given positive constants and
$\bfsigma:\Gamma\times (0,T)\rightarrow \mathbb{R}^2$, $j=1,2$,
$$\bfsigma\in{W^{2,2}(0,T;(\mathbf{W}^{1/2,2}_{\Gamma})^*)}\cap
{W^{1,2}(0,T;\mathbf{W}^{1/2,2}_{\Gamma})}.$$
\end{itemize}

\subsection{Admissible domains}\label{Admissible domains}

In what follows, we assume that (cf. Figure~\ref{composite})
\begin{itemize}
\item[(i)] $\Omega$ is decomposed into  nonoverlapping subdomains
$\Omega_{\ell}$;
\item[(ii)] there exists a finite set
$\mathcal{S}\subset\partial\Omega$ of boundary points such that
$\partial\Omega\setminus \mathcal{S}$ is smooth (of class
$C^\infty$);
\item[(iii)] for every $P\in\mathcal{S}$ there exists a neighborhood
$\mathcal{U}_P$ and a diffeomorphism $D_P$ mapping $\Omega \cap
\mathcal{U}_P$ onto $\mathcal{K}_P\cap B_P$, where $\mathcal{K}_P$
is an angle of size $\omega_{P}<\pi$ with vertex at the origin
(shifted into $P$),
$$
\mathcal{K}_P:=\left\{[x_1,x_2]\in\mathbb{R}^2; \; 0<r<\infty, \; 0
< \varphi < \omega_{P} \right\},
$$
and $B_P$ is a unit circle centered at the origin ($r,\varphi$
denote the polar coordinates in the $(x'_1,x'_2)$-plane);
\item[(iv)] the interfaces $\Gamma_{m\ell}$ are smooth (of class
$C^\infty$), $m=1,\dots,M$, $m\neq\ell$;
\item[(v)] there are no cross points of $\overline{\Gamma}_{m\ell}$
in $\overline{\Omega}$.
\end{itemize}
Let $\mathcal{M}$ be the set of all boundary points
$A\in\Gamma\equiv\partial\Omega\cap\Gamma_{m\ell}$, $m=1,\dots,M$,
i.e. the points where any interface $\Gamma_{m\ell}$ crosses the
exterior boundary $\partial\Omega$.
\begin{figure}[h]
\centering
\includegraphics[width=5.0cm]{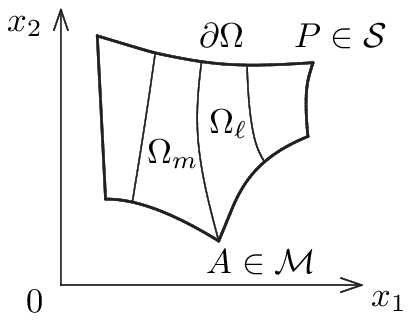}
%\caption{The admissible domain $\Omega$.}\label{composite}
\quad \includegraphics[width=6.0cm]{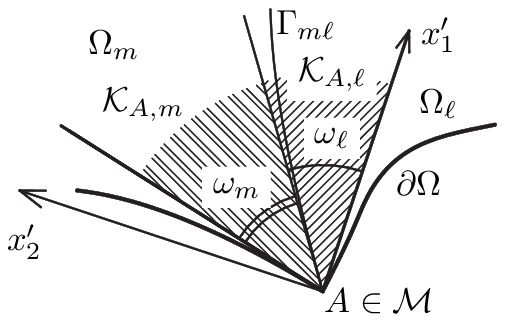}
\caption{Admissible domains.}\label{composite}
\end{figure}
Further, we assume that
\begin{itemize}
\item[(vi)]  for every $A\in\mathcal{M}$,
such that $A\in\partial\Omega\cap\Gamma_{m\ell}$, there exists a
neighborhood $\mathcal{U}_A$ and a diffeomorphism $D_{A,\ell}$ and
$D_{A,m}$, respectively, mapping $\Omega_{\ell} \cap \mathcal{U}_A$
onto $\mathcal{K}_{A,\ell}\cap B_A$ and $\Omega_m \cap
\mathcal{U}_A$ onto $\mathcal{K}_{A,m}\cap B_A$, respectively, where
$\mathcal{K}_{A,\ell}$ and $\mathcal{K}_{A,m}$, respectively, is an
angle of size $\omega_{\ell}$ and $\omega_m$, respectively, with
vertex at $A$
$$
\mathcal{K}_{A,\ell}:=\left\{[x_1,x_2]\in\mathbb{R}^2; \;
0<r<\infty, \; 0 < \varphi < \omega_{\ell} \right\}
$$
and
$$
\mathcal{K}_{A,m}:=\left\{[x_1,x_2]\in\mathbb{R}^2; \; 0<r<\infty,
\; \omega_{\ell}  < \varphi < \omega_{\ell}+\omega_{m} \right\},
$$
respectively, and $B_A$ is a unit circle centered at the origin
($r,\varphi$ denote the polar coordinates in the $(x'_1,x'_2)$-plane
with the origin at $A$);
\item[(vii)] for every $A\in\mathcal{M}$,
such that $A\in\partial\Omega\cap\Gamma_{m\ell}$ we have
$\omega_{\ell}=\omega_m$;
\item[(viii)] for every $A\in\mathcal{M}$,
$A\in\partial\Omega\cap\Gamma_{m\ell}$, we have
$\omega_{A}=\omega_{\ell}+\omega_m=2\omega_{\ell}\leq\pi$.
\end{itemize}

\JZtext{
\begin{rmk}
Note that conditions (i)--(viii) incorporate, as a special case, rectangular
domain composed of regular rectangles $\Omega_\ell$. This serves as a
basic model for building envelopes, e.g.~\cite{Kalagasidis:2007:IBPT}.
\end{rmk}
}
%-----------------------------------------------------------------------------

\section{Solutions to an Auxiliary Linearized System}\label{linearproblem}
%%%%%%%%%%%%%%%%%%%%%%%%%%%%%%%%%%%%%%%%%%%%%%%%%%%%%%%%%%%%%%%%%%%%%%%
Following the standard methodology of contraction-based proofs, we
consider first an auxiliary linear problem with homogeneous initial
condition in the form
\begin{align}
\beta^{ji}_{\ell}\frac{\partial u^i_{\ell}}{\partial t}-\nabla \cdot
( \kappa^{ji}_{\ell}\nabla u^i_{\ell} ) &= f^j_{\ell}(\bfx,t) &&
{\rm in } \; { Q_{\ell T}}, \label{eqlin01}
\\
\kappa^{ji}_{\ell}\frac{\partial u^i_{\ell}}{\partial
\bfn_{\ell}}+\nu_{\ell}^j u^j_{\ell}&=g^j_{\ell}(\bfx,t) && {\rm on}
\;
{ S_{\ell T}},\label{eqlin02}
\\
u^j_{\ell} &= u^j_{m} && {\rm on } \;
\Gamma_{m\ell}\times(0,T),\label{eqlin03}
\\
\kappa^{ji}_{\ell}\frac{\partial u^i_{\ell}}{\partial
\bfn_{\ell}}+\kappa^{ji}_{m}\frac{\partial u^i_{m}}{\partial
\bfn_{m}}&=0 && {\rm on} \;
\Gamma_{m\ell}\times(0,T),\label{eqlin04}
\\
\bfu_{\ell}(\bfx,0)&=\bf0 &&{\rm in } \;
\Omega_{\ell}.\label{eqlin05}
\end{align}

\paragraph{Assumptions} In
 \eqref{eqlin01}--\eqref{eqlin05}  $\nu^j_{\ell}$  are real positive constants,
$\beta^{ji}_{\ell}:=\beta^{ji}_{\ell}(\bfx)$,
$\kappa^{ji}_{\ell}:=\kappa^{ji}_{\ell}(\bfx)$ are real positive
Lipschitz continuous functions such that
\begin{equation}\label{lingen_par1}
\beta^{11}_{\ell}\beta^{22}_{\ell}\kappa^{12}_{\ell}\kappa^{21}_{\ell}>
\left(\frac{\beta^{12}_{\ell}\kappa^{21}_{\ell}+
\beta^{21}_{\ell}\kappa^{12}_{\ell}}{2}\right)^2 \quad   \textmd{ in
} { \overline{\Omega}_{\ell} }
\end{equation}
and { the ellipticity condition}
\begin{equation}\label{lingen_par2}
\kappa^{11}_{\ell}\kappa^{22}_{\ell}>\kappa^{12}_{\ell}\kappa^{21}_{\ell}
\quad   \textmd{ in } { \overline{\Omega}_{\ell}}.
\end{equation}

\begin{dfn}
Let $\bff_{\ell}\in \mathcal{Y}_{\ell,T}$ and $\bfg_{\ell}\in
{W^{2,2}(0,T;(\mathbf{W}^{1/2,2}_{\partial\Omega_{\ell}\cap\Gamma})^*)}\cap
{W^{1,2}(0,T;\mathbf{W}^{1/2,2}_{\partial\Omega_{\ell}\cap\Gamma})}$.
Then $\bfu_{\ell}\in \mathcal{X}_{\ell,T}$ is called a strong
solution to the system \eqref{eqlin01}--\eqref{eqlin05} iff
\begin{multline}\label{eq_lin01}
\sum_{\ell=1}^{M} \int_{\Omega_{\ell}}
\beta^{ji}_{\ell}\frac{\partial u^i_{\ell}}{\partial t} \, v^j \,
{\rm d}\bfx + \sum_{\ell=1}^{M} \int_{\Omega_{\ell}}
  \kappa^{ji}_{\ell}\nabla u^i_{\ell}\cdot \nabla v^j  \,
{\rm d}\bfx   +
\sum_{\ell=1}^{M}\int_{\partial\Omega_{\ell}\cap\Gamma} \nu^j_{\ell}
\; u^j_{\ell}\,v^j \, {\rm d}S
\\
= \sum_{\ell=1}^{M}\int_{\Omega_{\ell}} f^j_{\ell} \, v^j \, {\rm
d}\bfx + \sum_{\ell=1}^{M}\int_{\partial\Omega_{\ell}\cap\Gamma}
g^j_{\ell}\,v^j \, {\rm d}S
\end{multline}
holds for every $\bfv\in \mathbf{W}^{1,2}$ and almost every $t\in
(0,T)$.
\end{dfn}

\begin{rmk}\label{compatibility_conditions}
The regularity in time direction naturally imposes two higher order
compatibility conditions on the given functions $\bff_{\ell}$ and
$\bfg_{\ell}$ in \eqref{eqlin01}--\eqref{eqlin02}. Namely, the first
one requires $\bfg_{\ell}(\bfx,0)$ to be compatible with
\eqref{eqlin02} while the second one roughly says that
$\bfu_{\ell}'(t)|_{t=0}$ has to belong to appropriate Sobolev
spaces. This implies additional conditions on $\bff_{\ell}(\bfx,0)$
included in the definition of the space $\mathcal{Y}_{\ell,T}$.
\end{rmk}

\begin{thm}\label{thm_lin}
Let $\bff_{\ell}\in \mathcal{Y}_{\ell,T}$, $\bfg_{\ell}\in
{W^{2,2}(0,T;(\mathbf{W}^{1/2,2}_{\partial\Omega_{\ell}\cap\Gamma})^*)}\cap
{W^{1,2}(0,T;\mathbf{W}^{1/2,2}_{\partial\Omega_{\ell}\cap\Gamma})}$
and $\bfg_{\ell}(\bfx,0)=\bf0$ on $\partial\Omega$. Then there
exists the unique strong solution $\bfu_{\ell}\in
\mathcal{X}_{\ell,T}$ to the system \eqref{eqlin01}--\eqref{eqlin05}
and the following estimate holds
\begin{equation}\label{eq11a}
\|\bfu_{\ell}\|_{\mathcal{X}_{\ell,T}} \leq  c \left(
\|\bff_{\ell}\|_{\mathcal{Y}_{\ell,T}} +  \|
\bfg_{\ell}\|_{W^{2,2}(0,T;(\mathbf{W}^{1/2,2}_{\partial\Omega_{\ell}\cap\Gamma})^*)}
+
\|\bfg_{\ell}\|_{W^{1,2}(0,T;\mathbf{W}^{1/2,2}_{\partial\Omega_{\ell}\cap\Gamma})}\right).
\end{equation}
\end{thm}
To prove Theorem \ref{thm_lin} we prepare the following definitions
and lemmas.
\begin{dfn}[Problem ($P_f$)] Let us define Problem ($P_f$) by the
linear transmission system \eqref{eqlin01}--\eqref{eqlin05} with
$\bfg_{\ell}\equiv{\bf0}$ on $\partial\Omega_{\ell}\cap\Gamma
\times[0,T)$.
\end{dfn}
\begin{lem}\label{lm_sup_10}
Let $\bff_{\ell}\in \mathcal{Y}_{\ell,T}$. Then there exists the
unique strong solution $\bfu_{\ell}\in \mathcal{X}_{\ell,T}$ of
Problem $(P_f)$. Moreover, the following estimate holds
\begin{equation}\label{eq11b}
\|\bfu_{\ell}\|_{\mathcal{X}_{\ell,T}} \leq c
\|\bff_{\ell}\|_{\mathcal{Y}_{\ell,T}}.
\end{equation}
\end{lem}
\begin{proof}
See \ref{proof_1}. The proof relies on the results for stationary
transmission problem presented in \ref{regularity_stationary}.
\end{proof}
\begin{dfn}[Problem ($P_g$)] Let us define Problem ($P_g$) by the
linear transmission system \eqref{eqlin01}--\eqref{eqlin05} with
$\bff_{\ell}\equiv{\bf0}$ in $\Omega_{\ell} \times(0,T)$.
\end{dfn}
\begin{lem}\label{lm_sup_11}
Let $\bfg_{\ell}\in
{W^{2,2}(0,T;(\mathbf{W}^{1/2,2}_{\partial\Omega_{\ell}\cap\Gamma})^*)}\cap
{W^{1,2}(0,T;\mathbf{W}^{1/2,2}_{\partial\Omega_{\ell}\cap\Gamma})}$
and the compatibility condition $\bfg_{\ell}(\bfx,0)=\bf0$ on
$\partial\Omega$ be satisfied. Then there exists the unique strong
solution $\bfu_{\ell}\in \mathcal{X}_{\ell,T}$ of Problem $(P_g)$
and the following estimate holds
\begin{equation}\label{eq11c}
\|\bfu_{\ell}(t)\|_{\mathcal{X}_{\ell,T}}\leq c  \left( \|
\bfg_{\ell}\|_{W^{2,2}(0,T;(\mathbf{W}^{1/2,2}_{\partial\Omega_{\ell}\cap\Gamma})^*)}
+ \|
\bfg_{\ell}\|_{W^{1,2}(0,T;\mathbf{W}^{1/2,2}_{\partial\Omega_{\ell}\cap\Gamma})}
\right).
\end{equation}
\end{lem}
\begin{proof}
See \ref{proof_2}. Similarly as in the proof of Lemma
\ref{lm_sup_10}, we use the results for stationary transmission
problem presented in \ref{regularity_stationary}.
\end{proof}

\begin{proof}[Proof of Theorem \ref{thm_lin}]
The assertion follows from the superposition principle of the
solutions to the linear Problems ($P_f$) and ($P_g$).
\end{proof}

%-------------------------------------------------------------------------

\section{Solutions to the Nonlinear Parabolic System}\label{NonlinearProblem}

\begin{dfn}[Problem ($P_0$)] Let us define Problem ($P_0$) by the
initial--boundary value transmission system
\eqref{eq01}--\eqref{eq03} with data and structure conditions
 satisfying the assumptions (A1)--(A4), see Subsection
\ref{structure_conditions}.
\end{dfn}
\begin{dfn}
A function $\bfu_{\ell}$, such that $\bfu'_{\ell}(t) \in
L^2(0,T;\mathbf{W}^{2,2}_{\ell})$ and $\bfu''_{\ell}(t) \in
L^2(0,T;\mathbf{L}^2_{\ell})$, is called a strong solution of
Problem ($P_0$) on $(0,T)$ with initial data $\bfmu_{\ell}\in
\mathbf{W}^{3,2}_{\ell}$ iff
\begin{multline*}
\sum_{\ell=1}^{M} \int_{\Omega_{\ell}}
b^{ji}_{\ell}(\bfu_{\ell})\frac{\partial u^i_{\ell}}{\partial t} \,
v^j \, {\rm d}\bfx + \sum_{\ell=1}^{M} \int_{\Omega_{\ell}}
a^{ji}_{\ell}(\bfu_{\ell})\nabla u^i_{\ell}\cdot\nabla
u^j_{\ell}\;{\rm d}\bfx
\\
+ \sum_{\ell=1}^{M} \int_{\partial\Omega_{\ell}\cap\Gamma}
\alpha^j_{\ell}({u}^j_{\ell}-\sigma^j)\, v^j \; {\rm d}S =
\sum_{\ell=1}^{M}\int_{\Omega_{\ell}}f^j_{\ell}(\bfu_{\ell}) \, v^j
\, {\rm d}\bfx
\end{multline*}
holds for every $\bfv\in\mathbf{W}^{1,2}$ and almost every $t\in
(0,T)$ and
\begin{equation*}
\bfu_{\ell}(\bfx,0)=\bfmu_{\ell}(\bfx) \; \textmd{ in }
\Omega_{\ell}.
\end{equation*}
\end{dfn}

%-------------------------------------------------------------------------------

\begin{thm}[Main result]\label{main_R_1}
Let the assumptions {\rm (A1)--(A4)} be satisfied. For a given
$\bfmu_{\ell}\in \mathbf{W}^{3,2}_{\ell}$, which is supposed to be
compatible with \eqref{eq02}--\eqref{transmission_conditions_01},
there exists $T^*\in(0,T]$ and a function $\bfu_{\ell}$ such that
$\bfu_{\ell}$ is the strong solution of Problem {\rm($P_0$)} on
$(0,T^*)$.
\end{thm}
Proof of the main result is postponed to the end of this section. We
start from a related problem with homogeneous initial condition. To
that end, let $\bfu_{\ell}$ be the strong solution of Problem
($P_0$) on $(0,T)$, $ \bfu_{\ell}=\bfmu_{\ell}+\bfy_{\ell}$. Then
$\bfy_{\ell} \in \mathcal{X}_{\ell,T}$ and the following equation
\begin{multline*}
\sum_{\ell=1}^{M} \int_{\Omega_{\ell}}
b^{ji}_{\ell}(\bfmu_{\ell}+\bfy_{\ell})\frac{\partial
y^i_{\ell}}{\partial t} \, v^j \, {\rm d}\bfx + \sum_{\ell=1}^{M}
\int_{\Omega_{\ell}} a^{ji}_{\ell}(\bfmu_{\ell}+\bfy_{\ell})
\nabla(\mu^i_{\ell}+y^i_{\ell})\cdot\nabla v^j \;{\rm d}\bfx \\
+ \sum_{\ell=1}^{M} \int_{\partial\Omega_{\ell}\cap\Gamma}
\alpha^j_{\ell}(\mu^j_{\ell}+y^j_{\ell}-\sigma^j)\, v^j \; {\rm
d}S = \sum_{\ell=1}^{M}\int_{\Omega_{\ell}}
f^j_{\ell}(\bfmu_{\ell}+\bfy_{\ell}) \, v^j \, {\rm d}\bfx
\end{multline*}
holds for every $\bfv\in\mathbf{W}^{1,2}$ and almost every $t\in
(0,T)$. %and
%\begin{equation*}
%\bfy_{\ell}(\bfx,0)={\bf0} \; {\textmd{ in }} \Omega_{\ell}.
%\end{equation*}
This amounts to solving the problem with shifted data
\begin{eqnarray*}
\widehat{b}^{\,ji}_{\ell}(\bfx,\bfy_{\ell})&=&b^{\,ji}_{\ell}(\bfy_{\ell}+\bfmu_{\ell}),\\
\widehat{a}^{ji}_{\ell}(\bfx,\bfy_{\ell})&=&{a}^{ji}_{\ell}(\bfy_{\ell}+\bfmu_{\ell}),\\
\widehat{f}^{\,j}_{\ell}(\bfx,\bfy_{\ell})&=&f^j_{\ell}(\bfmu_{\ell}
+\bfy_{\ell}).
\end{eqnarray*}
We often omit the argument ``$\bfx$'' writing shortly
$\widehat{a}^{ji}_{\ell}(\bfy_{\ell})$ instead of
$\widehat{a}^{ji}_{\ell}(\bfx,\bfy_{\ell})$,
$\widehat{b}^{ji}_{\ell}(\bfy_{\ell})$ instead of
$\widehat{b}^{ji}_{\ell}(\bfx,\bfy_{\ell})$ and
$\widehat{f}^{j}_{\ell}(\bfy_{\ell})$ instead of
$\widehat{f}^{j}_{\ell}(\bfx,\bfy_{\ell})$.
\begin{dfn}\label{operator_K}
Define the operator $\mathscr{K}: \mathcal{X}_{\ell,T}\rightarrow
\mathcal{Y}_{\ell,T}$ given by
\begin{eqnarray}
\sum_{\ell=1}^{M}\int_{\Omega_{\ell}} \mathscr{K}(\bfphi_{\ell})
\cdot \bfv \; {\rm d}\bfx &=& \sum_{\ell=1}^{M} \int_{\Omega_{\ell}}
\left(  \widehat{b}^{\,ji}_{\ell}({\bf0}) -
\widehat{b}^{\,ji}_{\ell}(\bfphi_{\ell}) \right)\frac{\partial
\phi^i_{\ell}}{\partial t} \, v^j \, {\rm d}\bfx
\nonumber \\
&& + \sum_{\ell=1}^{M} \int_{\Omega_{\ell}} \left(
\widehat{a}^{\,ji}_{\ell}({\bf0}) -
\widehat{a}^{\,ji}_{\ell}(\bfphi_{\ell})\right) \nabla\phi^i_{\ell}
\cdot\nabla v^j \; {\rm d}\bfx
\nonumber\\
&&  - \sum_{\ell=1}^{M} \int_{\Omega_{\ell}}
\widehat{a}^{\,ji}_{\ell}(\bfphi_{\ell}) \nabla\mu^i_{\ell} \cdot
\nabla v^j \;{\rm d}\bfx
\nonumber\\
&&+ \sum_{\ell=1}^{M}\int_{\Omega_{\ell}}
\widehat{f}^{\,j}_{\ell}(\bfphi_{\ell}) \, v^j \, {\rm d}\bfx,
\label{eq230}
\end{eqnarray}
which holds for every $\bfv\in\mathbf{W}^{1,2}$ and almost every
$t\in(0,T)$.
\end{dfn}
\begin{rmk}\label{rem_homo}
Let $\bfu_{\ell}=\bfmu_{\ell}+\bfy_{\ell}$. The function
$\bfu_{\ell}$ is the strong solution of Problem $(P_0)$ on $(0,T)$
with initial data $\bfmu_{\ell}\in \mathbf{W}^{3,2}_{\ell}$ iff for
$\bfy_{\ell} \in \mathcal{X}_{\ell,T}$ the following equation
\begin{multline*}
\sum_{\ell=1}^{M} \int_{\Omega_{\ell}}
\widehat{b}^{\,ji}_{\ell}({\bf0})\frac{\partial y^i_{\ell}}{\partial
t} \, v^j \, {\rm d}\bfx + \sum_{\ell=1}^{M} \int_{\Omega_{\ell}}
\widehat{a}^{ji}_{\ell}({\bf0})\nabla y^i_{\ell}\cdot\nabla
v^j\;{\rm d}\bfx  + \sum_{\ell=1}^{M}
\int_{\partial\Omega_{\ell}\cap\Gamma}
\alpha^j_{\ell}y^j_{\ell}\, v^j \; {\rm d}S \\
+\sum_{\ell=1}^{M} \int_{\partial\Omega_{\ell}\cap\Gamma}
\alpha^j_{\ell}(\mu^j_{\ell}-\sigma^j)\, v^j \; {\rm d}S =
\sum_{\ell=1}^{M}\int_{\Omega_{\ell}} \mathscr{K}(\bfy_{\ell}) \cdot
\bfv \; {\rm d}\bfx
\end{multline*}
holds for every $\bfv\in\mathbf{W}^{1,2}$ and almost every $t\in
(0,T)$.
\end{rmk}
Before proceeding to the proof of the main result of the paper, we
prepare some auxiliary lemmas and propositions.

{ For a fixed $R>0$ define the closed ball $B_R(T) \subset
\mathcal{X}_{\ell,T}$
\begin{equation*}
B_R(T):= \left\{ \bfphi \in \mathcal{X}_{\ell,T}; \;
\|{\bfphi}\|_{\mathcal{X}_{\ell,T}} \leq R \right\} .
\end{equation*}}

\begin{lem}\label{lem010}
Let $\bfphi_{\ell} \in { B_R(T)}$. Then
\begin{equation}\label{eq600}
\left\| \mathscr{K}(\bfphi_{\ell})\right\|_{\mathcal{Y}_{\ell,T}}
\leq  c_1 C(T) \left(\|\bfphi_{\ell}\|^3_{\mathcal{X}_{\ell,T}} +
\|\bfphi_{\ell}\|^{2}_{\mathcal{X}_{\ell,T}}+\|\bfphi_{\ell}\|_{\mathcal{X}_{\ell,T}}\right)
+ c_2,
\end{equation}
where the function $C(T)\rightarrow 0_+$ for $T\rightarrow 0_+$ and
the constants $c_1,c_2>0$, both independent of $\bfphi_{\ell}$, do
not depend on $T$.
\end{lem}
\begin{proof} The proof is rather technical. To derive the estimate \eqref{eq600}
we extensively use the embeddings and estimates
\eqref{embedding_theorems}--\eqref{odh3}. First, for all
$\bfphi_{\ell} \in  { B_R(T)}$ we have
\begin{multline}\label{eq:split_estimate}
\left\| \mathscr{K}(\bfphi_{\ell})\right\|_{\mathcal{Y}_{\ell,T}}
\leq \left\| \left(  \widehat{\bfb}_{\ell}({\bf0}) -
\widehat{\bfb}_{\ell}(\bfphi_{\ell}) \right) \bfphi'_{\ell}(t)
\right\|_{\mathcal{Y}_{\ell,T}} \!\!+
\left\|\nabla\!\!\cdot\!\left[\left( \widehat{a}^{ji}_{\ell}({\bf0})
- \widehat{a}^{ji}_{\ell}(\bfphi_{\ell})\right)\!\nabla\phi^i_{\ell}
\right]\right\|_{\mathcal{Y}_{\ell,T}}
\\
 + \left\|\nabla\cdot\left[ \widehat{a}^{ji}_{\ell}(\bfphi_{\ell})
\nabla \mu^i_{\ell} \right] \right\|_{\mathcal{Y}_{\ell,T}}+\left\|
\widehat{\bff}_{\ell}(\bfphi_{\ell})
 \right\|_{\mathcal{Y}_{\ell,T}}.
\end{multline}
Now, we have to estimate each term on the right-hand side of
\eqref{eq:split_estimate}. Successively, we use \eqref{odh1} and
(A1) (see Subsection \ref{structure_conditions}) to estimate the
first term:
\begin{eqnarray}\label{est015}
\left\| \left(  \widehat{\bfb}_{\ell}({\bf0}) -
\widehat{\bfb}_{\ell}(\bfphi_{\ell}) \right) \bfphi'_{\ell}(t)
\right\|_{\mathcal{Y}_{\ell,T}} &\leq&  \left\| \left(
\widehat{\bfb}_{\ell}({\bf0}) - \widehat{\bfb}_{\ell}(\bfphi_{\ell})
\right) \bfphi''_{\ell}(t) \right\|_{L^2(Q_{\ell T})^2}
\nonumber \\
&&+ \left\| \frac{\partial \widehat{b}^{ji}_{\ell}(\bfphi_{\ell})}
{\partial \phi^l_{\ell}} (\phi^l_{\ell})'(t) (\phi^i_{\ell})'(t)
\right\|_{L^2(Q_{\ell T})^2}
\nonumber \\
&\leq&  c_1\|\bfphi_{\ell}\|_{L^{\infty}(Q_{\ell
T})^2}\|\bfphi''_{\ell}(t)\|_{L^2(Q_{\ell T})^2}
\nonumber \\
&&+ c_2 \|\bfphi'_{\ell}(t)\|_{L^4(Q_{\ell T})^2}
\nonumber \\
&\leq& c_1 T^{1/2}\,\|\bfphi_{\ell}\|^2_{\mathcal{X}_{\ell,T}} +
c_2T^{1/4}\,\|\bfphi_{\ell}'(t)\|^2_{L^{8}(0,T;\mathbf{L}^{4}_{\ell})}
\nonumber \\
&\leq& c_1 T^{1/2}\,\|\bfphi_{\ell}\|^2_{\mathcal{X}_{\ell,T}} +
c_2T^{1/4}\,\|\bfphi_{\ell}\|^2_{\mathcal{X}_{\ell,T}}.
\nonumber \\
\end{eqnarray}
Similarly, estimating the second term in \eqref{eq:split_estimate}
in the norm of the space $\mathcal{Y}_{\ell,T}$ we arrive at
\begin{eqnarray}\label{ineq010}
\left\|\nabla\cdot\left[ \left( \widehat{a}^{ji}_{\ell}({\bf0}) -
\widehat{a}^{ji}_{\ell}(\bfphi_{\ell}) \right)\nabla\phi^i_{\ell}
\right]\right\|_{\mathcal{Y}_{\ell,T}} \!\!\! &\!\!\! \leq& \left\|
\frac{\partial^2 \widehat{a}^{ji}_{\ell}(\bfphi_{\ell})}{\partial
\phi^l_{\ell} \partial \phi^r_{\ell}}(\phi^r_{\ell})'(t) \nabla
\phi^l_{\ell}\cdot\nabla\phi^i_{\ell}\right\|_{L^2(Q_{\ell T})^2}
\nonumber\\
& + & \left\|\frac{\partial
\widehat{a}^{ji}_{\ell}(\bfphi_{\ell})}{\partial \phi^l_{\ell} }
\nabla [(\phi^l_{\ell})'(t)]\cdot\nabla\phi^i_{\ell}
\right\|_{L^2(Q_{\ell T})^2}
\nonumber\\
& + & \left\|\frac{\partial^2
\widehat{a}^{ji}_{\ell}(\bfx,\bfphi_{\ell})}{\partial x^k \partial
\phi^l_{\ell} } (\phi^l_{\ell})'(t)
\frac{\partial\phi^i_{\ell}}{\partial x^k} \right\|_{L^2(Q_{\ell
T})^2}
\nonumber\\
& + & \left\|\frac{\partial
\widehat{a}^{ji}_{\ell}(\bfphi_{\ell})}{\partial \phi^l_{\ell} }
(\phi^l_{\ell})'(t)\Delta\phi^i_{\ell} \right\|_{L^2(Q_{\ell T})^2}
\nonumber\\
& + & \left\| \left[ \widehat{\bfa}_{\ell}({\bf0}) -
\widehat{\bfa}_{\ell}(\bfphi_{\ell}) \right] \Delta
\bfphi'_{\ell}(t) \right\|_{L^2(Q_{\ell T})^2}
\nonumber\\
& +& \left\|\nabla\!\! \left[\widehat{a}^{ji}_{\ell}({\bf0})
-\widehat{a}^{ji}_{\ell}(\bfx,\bfphi_{\ell}) \right]\!\cdot\! \nabla
(\phi^i_{\ell})'\right\|_{L^2(Q_{\ell T})^2}.
\nonumber\\
\end{eqnarray}
The first integral on the right hand side in \eqref{ineq010} can be
estimated
\begin{eqnarray*}
\left\| \frac{\partial^2
\widehat{a}^{ji}_{\ell}(\bfphi_{\ell})}{\partial \phi^l_{\ell}
\partial \phi^r_{\ell}}(\phi^r_{\ell})'(t) \nabla
\phi^l_{\ell}\cdot\nabla\phi^i_{\ell}\right\|^2_{L^2(Q_{\ell T})^2}
& \leq & c \int_0^T\left(\int_{\Omega_{\ell}}
\left|\bfphi'_{\ell}(t)\right|^2\left|\nabla\bfphi_{\ell}\right|^4
{\rm d}\bfx \right){\rm d}t
\nonumber \\
& \leq & c \int_0^T  \left\|
\bfphi'_{\ell}(t)\right\|^2_{\mathbf{L}^{4}_{\ell}} \left\|
\bfphi_{\ell} \right\|^4_{\mathbf{W}^{1,8}_{\ell}}  \, {\rm d} t
\nonumber \\
& \leq & c\left\| \bfphi'_{\ell}(t)
\right\|^{2}_{L^{2}(0,T;\mathbf{L}^{4}_{\ell})} \left\|\bfphi_{\ell}
\right\|^4_{L^{\infty}(0,T;\mathbf{W}^{1,8}_{\ell})}
\nonumber \\
&\leq&cT^{3/4}\!\!\left\|\bfphi'_{\ell}(t)
\right\|^2_{L^{8}(0,T;\mathbf{L}^{4}_{\ell})}
 T^{1/2}\!\!\left\|\bfphi_{\ell}\right\|^4_{\mathcal{X}_{\ell,T}}
\end{eqnarray*}
and applying \eqref{odh1} we obtain
\begin{equation}\label{est065}
\left\| \frac{\partial^2
\widehat{a}^{ji}_{\ell}(\bfphi_{\ell})}{\partial \phi^l_{\ell}
\partial \phi^r_{\ell}}(\phi^r_{\ell})'(t) \nabla
\phi^l_{\ell}\cdot\nabla\phi^i_{\ell}\right\|_{L^2(Q_{\ell T})^2}
\leq
 c \, T^{5/8}\,\|\bfphi_{\ell}\|^3_{\mathcal{X}_{\ell,T}}.
\end{equation}
Similarly
\begin{eqnarray*}
\left\|\frac{\partial
\widehat{a}^{ji}_{\ell}(\bfphi_{\ell})}{\partial \phi^l_{\ell} }
\nabla [(\phi^l_{\ell})'(t)]\cdot\nabla\phi^i_{\ell}
\right\|^2_{L^2(Q_{\ell T})^2} & \leq & c
\int_0^T\left(\int_{\Omega_{\ell}}
\left|\nabla\bfphi'_{\ell}(t)\right|^2\left|\nabla\bfphi_{\ell}\right|^2
{\rm d}\bfx\right){\rm d}t
\nonumber \\
& \leq & c \int_0^T  \left\| \bfphi'_{\ell}(t)
\right\|^2_{\mathbf{W}^{1,3}_{\ell}} \left\| \bfphi_{\ell}
\right\|^2_{\mathbf{W}^{1,6}_{\ell}} {\rm d}t
\nonumber \\
& \leq & c \left\|
\bfphi'_{\ell}(t)\right\|^{2}_{L^2(0,T;\mathbf{W}^{1,3}_{\ell})}
\left\|\bfphi_{\ell}\right\|^2_{L^{\infty}(0,T;\mathbf{W}^{1,6}_{\ell})}
\nonumber \\
& \leq & c T^{1/3}\!\!\left\|
\bfphi'_{\ell}(t)\right\|^{2}_{L^3(0,T;\mathbf{W}^{1,3}_{\ell})}
 T^{1/2}\!\!\left\|\bfphi_{\ell}\right\|^2_{\mathcal{X}_{\ell,T}}
\end{eqnarray*}
and using \eqref{odh3} we get
\begin{equation}\label{est066}
\left\|\frac{\partial
\widehat{a}^{ji}_{\ell}(\bfphi_{\ell})}{\partial \phi^l_{\ell} }
\nabla [(\phi^l_{\ell})'(t)]\cdot\nabla\phi^i_{\ell}
\right\|_{L^2(Q_{\ell T})^2}  \leq
 c \, T^{5/12}\,\|\bfphi_{\ell}\|^2_{\mathcal{X}_{\ell,T}}.
\end{equation}
Similarly, the third term  in \eqref{ineq010} can be estimated as
\begin{equation}\label{est063}
\left\|\frac{\partial^2
\widehat{a}^{ji}_{\ell}(\bfx,\bfphi_{\ell})}{\partial x^k \partial
\phi^l_{\ell} } (\phi^l_{\ell})'(t)
\frac{\partial\phi^i_{\ell}}{\partial x^k} \right\|_{L^2(Q_{\ell
T})^2}  \leq
 c \, T^{5/12}\,\|\bfphi_{\ell}\|^2_{\mathcal{X}_{\ell,T}}.
\end{equation}
Further
\begin{eqnarray}\label{est060}
 \left\|\frac{\partial
\widehat{a}^{ji}_{\ell}(\bfphi_{\ell})}{\partial \phi^l_{\ell} }
(\phi^l_{\ell})'(t)\Delta\phi^i_{\ell} \right\|^2_{L^2(Q_{\ell
T})^2} & \leq & c \int_0^T\left(\int_{\Omega_{\ell}}
\left|\bfphi'_{\ell}(t)\right|^2\left|\Delta\bfphi_{\ell}\right|^2
\;{\rm d}\bfx \right) \, {\rm d} t
\nonumber \\
& \leq & c \int_0^T  \left\|
\bfphi'_{\ell}(t)\right\|^2_{\mathbf{L}^{\infty}_{\ell}} \left\|
\bfphi_{\ell} \right\|^2_{\mathbf{W}^{2,2}_{\ell}}{\rm d}t
\nonumber \\
& \leq & c\left\| \bfphi'_{\ell}(t)
\right\|^{2}_{L^{2}(0,T;\mathbf{L}^{\infty}_{\ell})}
\left\|\bfphi_{\ell}\right\|^2_{L^{\infty}(0,T;\mathbf{W}^{2,2}_{\ell})}
\nonumber \\
& \leq & c T^{1/3}\!\!\left\|
\bfphi'_{\ell}(t)\right\|^{2}_{L^3(0,T;\mathbf{L}^{\infty}_{\ell})}
 T^{1/2}\!\!\left\|\bfphi_{\ell}\right\|^2_{\mathcal{X}_{\ell,T}}.
 \quad
\end{eqnarray}
Now~\eqref{est060} and \eqref{odh3} imply
\begin{equation}\label{est067}
 \left\|\frac{\partial
\widehat{a}^{ji}_{\ell}(\bfphi_{\ell})}{\partial \phi^l_{\ell} }
(\phi^l_{\ell})'(t)\Delta\phi^i_{\ell} \right\|_{L^2(Q_{\ell T})^2}
\leq
 c \, T^{5/12}\,\|\bfphi_{\ell}\|^2_{\mathcal{X}_{\ell,T}}.
\end{equation}
\eqref{odh2} yields the estimate
\begin{eqnarray}\label{est068}
\left\| \left[ \widehat{\bfa}_{\ell}({\bf0}) -
\widehat{\bfa}_{\ell}(\bfphi_{\ell}) \right] \Delta
\bfphi'_{\ell}(t)\right\|_{L^2(Q_{\ell T})^2}
 &\leq&
 c \|\bfphi_{\ell}\|_{_{L^{\infty}(Q_{\ell T})^2}}
\|\Delta\bfphi'_{\ell}(t)\|_{L^2(Q_{\ell T})^2}
\nonumber \\
&\leq&  cT^{1/2}\|\bfphi_{\ell}\|_{\mathcal{X}_{\ell,T}}
\|\Delta\bfphi'_{\ell}(t)\|_{L^2(Q_{\ell T})^2}
\nonumber \\
&\leq& cT^{1/2}\|\bfphi_{\ell}\|^2_{\mathcal{X}_{\ell,T}},
\end{eqnarray}
where we have used the Lipschitz continuity of
$\widehat{\bfa}_{\ell}$. In the similar way one can deduce
\begin{multline}\label{est068b}
\left\|\nabla\!\! \left[\widehat{a}^{ji}_{\ell}({\bf0})
-\widehat{a}^{ji}_{\ell}(\bfphi_{\ell}) \right]\!\cdot\! \nabla
(\phi^i_{\ell})'(t)\right\|_{L^2(Q_{\ell T})^2}
\\
\leq
c_1\|\nabla\bfphi'_{\ell}(t)\|_{L^{2}(0,T;\mathbf{L}^{2}_{\ell})}+
c_2
\left\|\bfphi_{\ell}\right\|_{L^{\infty}(0,T;\mathbf{W}^{1,6}_{\ell})}
\left\| \bfphi'_{\ell}(t)\right\|_{L^2(0,T;\mathbf{W}^{1,3}_{\ell})}
\\
\leq c_1 T^{1/6}\left\|
\bfphi'_{\ell}(t)\right\|_{L^3(0,T;\mathbf{W}^{1,3}_{\ell})}  +c_2
 T^{1/4}\left\|\bfphi_{\ell}\right\|_{\mathcal{X}_{\ell,T}}
 T^{1/6}\left\|
\bfphi'_{\ell}(t)\right\|_{L^3(0,T;\mathbf{W}^{1,3}_{\ell})}
\\
\leq c_1 T^{1/6}\|\bfphi_{\ell}\|_{\mathcal{X}_{\ell,T}}+ c_2
T^{5/12}\|\bfphi_{\ell}\|^2_{\mathcal{X}_{\ell,T}}.
\end{multline}
Finally, the estimates \eqref{ineq010}--\eqref{est068b} imply
\begin{multline}\label{est070}
\left\|\nabla\cdot\left[ \left( \widehat{a}^{ji}_{\ell}({\bf0}) -
\widehat{a}^{ji}_{\ell}(\bfphi_{\ell}) \right)\nabla\phi^i_{\ell}
\right]\right\|_{\mathcal{Y}_{\ell,T}}
\\
 \leq  c C(T)
\left( \|\bfphi_{\ell}\|^3_{\mathcal{X}_{\ell,T}} +
\|\bfphi_{\ell}\|^2_{\mathcal{X}_{\ell,T}}+
\|\bfphi_{\ell}\|_{\mathcal{X}_{\ell,T}}\right),
\end{multline}
where the function $C(T)\rightarrow 0_+$ for $T\rightarrow 0_+$ and
$c$ is independent of $T$.

Further, estimating the third term on the right hand side in
\eqref{eq:split_estimate} one obtains
\begin{eqnarray}\label{est080}
\left\|\nabla\cdot\left[ \widehat{a}^{ji}_{\ell}(\bfphi_{\ell})
\nabla \mu^i_{\ell} \right] \right\|_{\mathcal{Y}_{\ell,T}} &\leq&
\left\| \frac{\partial^2
\widehat{a}^{ji}_{\ell}(\bfphi_{\ell})}{\partial \phi^l_{\ell}
\partial \phi^r_{\ell}}(\phi^r_{\ell})'(t) \nabla
\phi^l_{\ell}\cdot\nabla\mu^i_{\ell}\right\|_{L^2(Q_{\ell T})^2}
\nonumber\\
&+&\left\|\frac{\partial
\widehat{a}^{ji}_{\ell}(\bfphi_{\ell})}{\partial \phi^l_{\ell} }
\nabla [(\phi^l_{\ell})'(t)]\cdot\nabla\mu^i_{\ell}
\right\|_{L^2(Q_{\ell T})^2}
\nonumber\\
&+& \left\|\frac{\partial
\widehat{a}^{ji}_{\ell}(\bfphi_{\ell})}{\partial \phi^l_{\ell} }
(\phi^l_{\ell})'(t)\Delta\mu^i_{\ell} \right\|_{L^2(Q_{\ell T})^2}
 \nonumber\\
&+& \left\|\frac{\partial^2
\widehat{a}^{ji}_{\ell}(\bfx,\bfphi_{\ell})}{\partial x^k \partial
\phi^l_{\ell} } (\phi^l_{\ell})'(t)
\frac{\partial\mu^i_{\ell}}{\partial x^k} \right\|_{L^2(Q_{\ell T})}
\nonumber\\
&+&\left\|\nabla\cdot\left[ \widehat{a}^{ji}_{\ell}({\bf0}) \nabla
\mu^i_{\ell} \right] \right\|_{\mathbf{W}^{1,2}_{\ell}}.
\end{eqnarray}
Estimating each term on the right hand side we arrive at
\begin{multline}\label{est081}
\left\| \frac{\partial^2
\widehat{a}^{ji}_{\ell}(\bfphi_{\ell})}{\partial \phi^l_{\ell}
\partial \phi^r_{\ell}}(\phi^r_{\ell})'(t) \nabla
\phi^l_{\ell}\cdot\nabla\mu^i_{\ell}\right\|^2_{L^2(Q_{\ell T})^2}
\leq c \!\! \int_0^T \!\!\! \left\| \bfphi'_{\ell}(t)
\right\|^2_{\mathbf{L}^{4}_{\ell}}\left\| \nabla\bfphi_{\ell}
\right\|^2_{\mathbf{L}^{8}_{\ell}} \left\| \nabla\bfmu_{\ell}
\right\|^2_{\mathbf{L}^{8}_{\ell}}{\rm d}t
\\
\leq c T^{3/4}
\|\bfphi'_{\ell}(t)\|^2_{L^8(0,T;\mathbf{L}^{4}_{\ell})}
\|\bfphi_{\ell}\|^2_{L^{\infty}(0,T;\mathbf{W}^{1,8}_{\ell})}
\left\| \nabla\bfmu_{\ell} \right\|^2_{\mathbf{L}^{8}_{\ell}}
\\
\leq  c T^{5/4}\|\bfphi_{\ell}\|^4_{\mathcal{X}_{\ell,T}},
\end{multline}
where we have used the estimate
$\|\bfphi_{\ell}\|^2_{L^{\infty}(0,T;\mathbf{W}^{1,8}_{\ell})}\leq
T^{1/2}\|\bfphi_{\ell}\|^2_{\mathcal{X}_{\ell,T}}$. Further
\begin{eqnarray}\label{est082}
\left\|\frac{\partial
\widehat{a}^{ji}_{\ell}(\bfphi_{\ell})}{\partial \phi^l_{\ell} }
\nabla [(\phi^l_{\ell})'(t)]\cdot\nabla\mu^i_{\ell}
\right\|^2_{L^2(Q_{\ell T})^2} &\leq& c \int_0^T \left\|
\nabla\bfphi'_{\ell}(t) \right\|^2_{\mathbf{L}^{3}_{\ell}}\left\|
\nabla\bfmu_{\ell} \right\|^2_{\mathbf{L}^{6}_{\ell}} \, {\rm d} t
\nonumber \\
& \leq & c \, T^{1/3}
\|\bfphi'_{\ell}(t)\|^2_{L^{3}(0,T;\mathbf{W}^{1,3}_{\ell})} \left\|
\nabla\bfmu_{\ell} \right\|^2_{\mathbf{L}^{6}_{\ell}},
\nonumber \\
\end{eqnarray}
\begin{eqnarray}\label{est083}
\left\|\frac{\partial
\widehat{a}^{ji}_{\ell}(\bfphi_{\ell})}{\partial \phi^l_{\ell} }
(\phi^l_{\ell})'(t)\Delta\mu^i_{\ell} \right\|^2_{L^2(Q_{\ell T})^2}
&\leq& c \int_0^T \left\|\bfphi'_{\ell}(t)
\right\|^2_{\mathbf{L}^{\infty}_{\ell}} \left\|\Delta\bfmu_{\ell}
\right\|^2_{\mathbf{L}^{2}_{\ell}}{\rm d}t
\nonumber \\
&\leq& c \, T^{1/3}
\|\bfphi'_{\ell}(t)\|^2_{L^{3}(0,T;\mathbf{L}^{\infty}_{\ell})}
\left\| \Delta\bfmu_{\ell} \right\|^2_{\mathbf{L}^{2}_{\ell}}.
\nonumber \\
\end{eqnarray}
and finally
\begin{eqnarray}\label{est083b}
\left\|\frac{\partial^2
\widehat{a}^{ji}_{\ell}(\bfx,\bfphi_{\ell})}{\partial x^k \partial
\phi^l_{\ell} } (\phi^l_{\ell})'(t)
\frac{\partial\mu^i_{\ell}}{\partial x^k} \right\|^2_{L^2(Q_{\ell
T})^2} &\leq& c \int_0^T \left\|\bfphi'_{\ell}(t)
\right\|^2_{\mathbf{L}^{\infty}_{\ell}} \left\|\nabla\bfmu_{\ell}
\right\|^2_{\mathbf{L}^{2}_{\ell}}{\rm d}t
\nonumber \\
&\leq& c \, T^{1/3}
\|\bfphi'_{\ell}(t)\|^2_{L^{3}(0,T;\mathbf{L}^{\infty}_{\ell})}
\left\|\nabla\bfmu_{\ell}\right\|^2_{\mathbf{L}^{2}_{\ell}}.
\nonumber \\
\end{eqnarray}
The inequalities \eqref{est080}--\eqref{est083b} yield the estimate
\begin{equation}\label{est084}
\left\|\nabla\cdot\left[ \widehat{a}^{ji}_{\ell}(\bfphi_{\ell})
\nabla \mu^i_{\ell} \right] \right\|_{\mathcal{Y}_{\ell,T}} \leq
c_1T^{5/8}\,\|\bfphi_{\ell}\|^2_{\mathcal{X}_{\ell,T}}+c_2 \,
T^{1/6}\, \|\bfphi_{\ell}\|_{\mathcal{X}_{\ell,T}}+c_3,
\end{equation}
where the constant $c_3$ in \eqref{est084} bounds the last term in
\eqref{est080}. Finally, taking into account (A3), the source term
can be estimated as
\begin{eqnarray}\label{est090}
\left\| \widehat{\bff}_{\ell}(\bfphi_{\ell})
\right\|_{\mathcal{Y}_{\ell,T}} &=&  \left\|\frac{\partial
\widehat{f}^{j}_{\ell}(\bfphi_{\ell})}{\partial \phi^l_{\ell} }
(\phi^l_{\ell})'(t) \right\|_{L^2(Q_{\ell T})^2}+
\left\|\widehat{\bff}_{\ell}({\bf0})
\right\|_{\mathbf{W}^{1,2}_{\ell}}
\nonumber \\
& \leq &   c T^{1/6}
\|\bfphi'_{\ell}\|_{L^{3}(0,T;\mathbf{L}^{2}_{\ell})}+c_2
\nonumber \\
& \leq &   c T^{1/6} \|\bfphi_{\ell}\|_{\mathcal{X}_{\ell,T}}+c_2.
\end{eqnarray}
Altogether, the estimates \eqref{est015}, \eqref{est070},
\eqref{est084} and \eqref{est090} yield the inequality
\eqref{eq600}.
\end{proof}

%--------------------------------------------------------------------------------

\begin{lem}\label{lem011}
There exists a nondecreasing function $c(R)$ ($c(R)$ does not depend
on $T$, $\bfphi_{\ell}$ and $\widetilde{\bfphi}_{\ell}$) such that
for all $\bfphi_{\ell}, \widetilde{\bfphi}_{\ell} \in B_R(T)$
\begin{equation}\label{ineq030}
\| \mathscr{K}(\bfphi_{\ell})-
\mathscr{K}(\widetilde{\bfphi}_{\ell})\|_{\mathcal{Y}_{\ell,T}}\leq
 c(R) \, C(T) \|\bfphi_{\ell}-
 \widetilde{\bfphi}_{\ell}\|_{\mathcal{X}_{\ell,T}},
\end{equation}
where the function $C(T)\rightarrow 0_+$ for $T\rightarrow 0_+$.
\end{lem}
\begin{proof}[Sketch of the proof] Similarly to Lemma~\ref{lem010}, the proof is
rather technical. Therefore, we only sketch the procedure and omit
the detailed derivations. First, we estimate
\begin{multline}\label{eq450}
\| \mathscr{K}(\bfphi_{\ell})-
\mathscr{K}(\widetilde{\bfphi}_{\ell})\|_{\mathcal{Y}_{\ell,T}}\leq\left\|
\left(  \widehat{\bfb}_{\ell}({\bf0}) -
\widehat{\bfb}_{\ell}(\bfphi_{\ell}) \right) \bfphi'_{\ell}(t)-
\left(  \widehat{\bfb}_{\ell}({\bf0}) -
\widehat{\bfb}_{\ell}(\widetilde{\bfphi}_{\ell}) \right)
\widetilde{\bfphi}'_{\ell}(t) \right\|_{\mathcal{Y}_{\ell,T}}
\\
+ \left\|\nabla\!\!\cdot\!\left[\left(
\widehat{a}^{ji}_{\ell}({\bf0}) -
\widehat{a}^{ji}_{\ell}(\bfphi_{\ell})\right)\!\nabla\phi^i_{\ell}
\right] - \nabla\!\!\cdot\!\left[\left(
\widehat{a}^{ji}_{\ell}({\bf0}) -
\widehat{a}^{ji}_{\ell}(\widetilde{\bfphi}_{\ell})\right)\!\nabla\widetilde{\phi}^i_{\ell}
\right] \right\|_{\mathcal{Y}_{\ell,T}}
\\
+ \left\|\nabla\cdot\left[ \widehat{a}^{ji}_{\ell}(\bfphi_{\ell})
\nabla \mu^i_{\ell} \right] - \nabla\cdot\left[
\widehat{a}^{ji}_{\ell}(\widetilde{\bfphi}_{\ell}) \nabla
\mu^i_{\ell} \right] \right\|_{\mathcal{Y}_{\ell,T}} + \left\|
\widehat{\bff}_{\ell}(\bfphi_{\ell})-\widehat{\bff}_{\ell}(\widetilde{\bfphi}_{\ell})
\right\|_{\mathcal{Y}_{\ell,T}}.
\end{multline}
The right hand side in \eqref{eq450} can be further estimated by
\begin{multline}\label{eq451}
\left\|\widehat{\bfb}_{\ell}({\bf0})\left( \bfphi'_{\ell}(t)
-\widetilde{\bfphi}'_{\ell}(t)
\right)\right\|_{\mathcal{Y}_{\ell,T}} +\left\|  \left(
\widehat{\bfb}_{\ell}(\bfphi_{\ell}) -
\widehat{\bfb}_{\ell}(\widetilde{\bfphi}_{\ell}) \right)
\bfphi'_{\ell}(t)\right\|_{\mathcal{Y}_{\ell,T}}
\\
+\left\| \widehat{\bfb}_{\ell}(\widetilde{\bfphi}_{\ell})\left(
\bfphi'_{\ell}(t) -\widetilde{\bfphi}'_{\ell}(t)
\right)\right\|_{\mathcal{Y}_{\ell,T}} + \left\|
\nabla\!\cdot\!\left[
\widehat{a}^{ji}_{\ell}({\bf0})\nabla\!\left(\phi^i_{\ell}-\widetilde{\phi}^i_{\ell}\right)
\right]\right\|_{\mathcal{Y}_{\ell,T}}
\\
\left\|\nabla\!\cdot\!\left[ \widehat{a}^{ji}_{\ell}({\bfphi})
\nabla\!\left({\phi}^i_{\ell}-\widetilde{\phi}^i_{\ell}\right)
\right]\right\|_{\mathcal{Y}_{\ell,T}} + \left\|
\nabla\!\cdot\!\left[\left( \widehat{a}^{ji}_{\ell}({\bfphi}) -
\widehat{a}^{ji}_{\ell}(\widetilde{\bfphi}_{\ell})\right)\!\nabla\widetilde{\phi}^i_{\ell}
\right]\right\|_{\mathcal{Y}_{\ell,T}}
\\
+
\left\|\nabla\!\cdot\!\left[\left(\widehat{a}^{ji}_{\ell}(\bfphi_{\ell})-
\widehat{a}^{ji}_{\ell}(\widetilde{\bfphi}_{\ell}) \right) \nabla
\mu^i_{\ell} \right]\right\|_{\mathcal{Y}_{\ell,T}} + \left\|
\widehat{\bff}_{\ell}(\bfphi_{\ell})-\widehat{\bff}_{\ell}
(\widetilde{\bfphi}_{\ell})
 \right\|_{\mathcal{Y}_{\ell,T}}.
\end{multline}
Estimating each term in \eqref{eq451} using the same arguments as in
the proof of Lemma \ref{lem010} and the assumptions (A1)--(A3), see
Section \ref{structure_conditions}, one obtains the inequality
\begin{equation}\label{est120}
\| \mathscr{K}(\bfphi_{\ell})-
\mathscr{K}(\widetilde{\bfphi}_{\ell})\|_{\mathcal{Y}_{\ell,T}} \leq
\underbrace{c_1\left(
R^2+R+1\right)}_{c(R)}C(T)\|\bfphi_{\ell}-\widetilde{\bfphi}_{\ell}
\|_{\mathcal{X}_{\ell,T}}
\end{equation}
for all $\bfphi_{\ell}, \widetilde{\bfphi}_{\ell} \in B_R(T)$. Now
\eqref{est120} yields \eqref{ineq030}.
\end{proof}

Using Definition \ref{operator_K} and Theorem \ref{thm_lin} we can
formulate the following
\begin{dfn}
Let $\mathscr{L}:\mathcal{X}_{\ell,T} \rightarrow
\mathcal{X}_{\ell,T}$ be an operator such that
$\mathscr{L}(\bfphi_{\ell})=\bfy_{\ell}$, if and only if
\begin{multline*}
\sum_{\ell=1}^{M} \int_{\Omega_{\ell}}
\widehat{b}^{\,ji}_{\ell}({\bf0})\frac{\partial y^i_{\ell}}{\partial
t} \, v^j \, {\rm d}\bfx + \sum_{\ell=1}^{M} \int_{\Omega_{\ell}}
\widehat{a}^{ji}_{\ell}({\bf0})\nabla y^i_{\ell}\cdot\nabla v^j \,
{\rm d}\bfx  + \sum_{\ell=1}^{M}
\int_{\partial\Omega_{\ell}\cap\Gamma} \alpha^j_{\ell}y^j_{\ell}\,
v^j \; {\rm d}S
\\
+\sum_{\ell=1}^{M} \int_{\partial\Omega_{\ell}\cap\Gamma}
\alpha^j_{\ell}(\mu^j_{\ell}-\sigma^j)\, v^j \; {\rm d}S =
\sum_{\ell=1}^{M}\int_{\Omega_{\ell}} \mathscr{K}(\bfphi_{\ell})
\cdot \bfv \; {\rm d}\bfx
\end{multline*}
holds for every $\bfv \in \mathbf{W}^{1,2}$ and almost every $t\in
(0,T)$ and $\bfy_{\ell}(\bfx,0)={\bf0} \; {\textmd{ in }}
\Omega_{\ell}$.
\end{dfn}
\begin{proof}[Proof of Theorem \ref{main_R_1}  (Main result).]
The proof of the main result is based on the Banach fixed point
theorem. Lemma \ref{lem010} and the estimate \eqref{eq11a} imply the
inequality
\begin{eqnarray}\label{est020}
\|\mathscr{L}(\bfphi_{\ell})\|_{\mathcal{X}_{\ell,T}}  &\leq& c_1 \,
\|\mathscr{K}(\bfphi_{\ell})\|_{\mathcal{Y}_{\ell,T}}+K_1 \nonumber \\
&\leq& c_2 C(T)\left(\|\bfphi_{\ell}\|^3_{\mathcal{X}_{\ell,T}} +
\|\bfphi_{\ell}\|^{2}_{\mathcal{X}_{\ell,T}}
+\|\bfphi_{\ell}\|_{\mathcal{X}_{\ell,T}}\right) + K_2
\end{eqnarray}
for all $\bfphi_{\ell}\in { B_R(T)}$, where $K_1$ and $K_2$ are
positive nondecreasing functions with respect to $T$ and independent
of $\bfphi_{\ell}$ and the constants $c_1,c_2$ are independent of
$\bfphi_{\ell}$ and $T$.  Further, linearity of
\eqref{eqlin01}--\eqref{eqlin05}, the estimate \eqref{eq11a} and
Lemma \ref{lem011} imply
\begin{multline}\label{est021}
\|\mathscr{L}(\bfphi_{\ell})-\mathscr{L}(\widetilde{\bfphi}_{\ell})
\|_{\mathcal{X}_{\ell,T}} \leq c_1\|\mathscr{K}(\bfphi_{\ell})-
\mathscr{K}(\widetilde{\bfphi}_{\ell})\|_{\mathcal{Y}_{\ell,T}}\\
 \leq  c(R) \, C(T)
 \|\bfphi_{\ell}-\widetilde{\bfphi}_{\ell}\|_{\mathcal{X}_{\ell,T}}
 \; \textmd{ for all } \bfphi_{\ell},\, \widetilde{\bfphi}_{\ell} \in B_R(T),
\end{multline}
where $c(R)$ is some nondecreasing function and  $C(T)\rightarrow
0_+$ for $T\rightarrow 0_+$. Now \eqref{est020} and \eqref{est021}
imply that for sufficiently small $T^*\in (0,T]$ there exists $R>0$
such that $\mathscr{L}:\mathcal{X}_{\ell,T^*} \rightarrow
\mathcal{X}_{\ell,T^*}$ maps $B_R(T^*)$ into itself and
$\mathscr{L}$ is a strict contraction in $B_R(T^*)$. Hence, using
the contraction mapping principle we have the existence of a fixed
point $\bfy_{\ell} \in B_R(T^*) \subset \mathcal{X}_{\ell,T^*}$,
such that $\mathscr{L}(\bfy_{\ell})=\bfy_{\ell}$. $\bfy_{\ell}$ is
uniquely determined in the ball $B_R(T^*)$. Set
$\bfu_{\ell}=\bfmu_{\ell}+\bfy_{\ell}$. By Remark \ref{rem_homo} the
function $\bfu_{\ell}$ is the strong solution of Problem ($P_0$) on
$(0,T^*)$.
\end{proof}

%--------------------------------------------------------------------------------

\section{Applications}\label{sec:applications}
%%%%%%%%%%%%%%%%%%%%%%
In this Section, we present examples of the coefficients of the
parabolic system~\eqref{eq01} related to models of transport in
porous media. Note that for brevity, we omit the subscript $\ell$
and the dependence of all variables on $\bfx$ and $t$ in what
follows.

All available engineering models of simultaneous heat and moisture
transfer possess a common structure, derived from two balance
equations of heat and mass~\cite{Hens:2007}:
\begin{eqnarray}\label{eq:heat_mass_structure}
\frac{\de H}{\de t} = - \nabla \cdot \bfj_Q + Q, && \frac{\de M}{\de
t} = - \nabla \cdot \bfj_m,
\end{eqnarray}
where $H$~(Jm$^{-3}$) is the specific enthalpy, $M$~(kgm$^{-3}$)
denotes the partial moisture density, $Q$~(Jm$^{-3}$s$^{-1}$) stands
for the intensity of internal heat sources and
$\bfj_Q$~(Jm$^{-2}$s$^{-1}$) and $\bfj_m$~(kgm$^{-2}$s$^{-1}$) are
the heat and moisture fluxes, respectively. This structure is also
reflected in the choice of the unknowns $\bfu$, which consist of the
temperature $u^1 = \theta$~($K$) and a quantity related to the
moisture content.

Individual models are then generated by the choice of the second
state variable $u^2$ and of the individual components in
system~\eqref{eq:heat_mass_structure}. In
Sections~\ref{sec:applications_kiessl}
and~\ref{sec:applications_kunzel}, following the expositions of
Dal\'{\i}k \emph{et al.}~\cite{DalikDanecekStastnik1996} and
\v{C}ern\'{y} and Rovnan\'{\i}kov\'{a}~\cite{CernyRovnanikova2002},
we briefly introduce two such representatives due to
Kiessl~\cite{Kiessl1983} and K\"{u}nzel~\cite{Kunzel:1995:SHM},
simplified by assuming that freezing of water in pores has a
negligible effect. An interested reader is referred
to~\cite{CernyRovnanikova2002,DalikDanecekStastnik1996,KuKi1997} for
additional discussion of the models and
to~\cite{CernyRovnanikova2002} for details on the terminology used
hereafter.

\subsection{The Kiessl model}\label{sec:applications_kiessl}
%%%%%%%%%%%%%%%%%%%%%%%%%%%%%

The enthalpy term in the Kiessl model is postulated in the form
\begin{equation}\label{eq:kiessl_enthalpy}
H = \rho_0 c_0 \theta + \rho_w c_w w \theta,
\end{equation}
where $\rho_0$~(kgm$^{-3}$) and $c_0$~(Jkg$^{-1}$K$^{-1}$) denote
the partial density and the specific heat capacity of the dry porous
matrix, $\rho_w$ and $c_w$ are analogous quantities for water and
$w$~(-) is the relative moisture content by volume. The heat flux
follows from the Fourier law for isotropic materials
\begin{equation}\label{eq:kiessl_heat_flux}
\bfj_Q = -\lambda( w, \theta ) \nabla \theta,
\end{equation}
with $\lambda$~(Jm$^{-1}$s$^{-1}$K$^{-1}$) being the state-dependent
coefficient of thermal conductivity.

The moisture balance is based on the moisture density provided by
\begin{equation}
M = \rho_w w + (e - w) \varphi \rho_{p,s}( \theta )
\end{equation}
where $e \geq w$~(-) denotes the porosity, $\varphi$ is the relative
humidity and $\rho_{p,s} \leq \rho_w$~(kgm$^{-3}$) is the
material-independent partial density of the saturated vapor phase,
given as a smooth increasing function of $\theta$. Assuming again
isotropy of the material, the corresponding moisture flux is
expressed in the form
\begin{equation}
\bfj_m = -\bigl( D_w( w, \theta ) \nabla w + D_\varphi( w, \theta )
\nabla \varphi + D_\theta( w, \theta  ) \nabla \theta \bigr),
\end{equation}
where $D_w$~(kgm$^{-1}$s$^{-1}$), $D_\varphi$~(kgm$^{-1}$s$^{-1}$)
and $D_\theta$~(kgm$^{-1}$K$^{-1}$s$^{-1}$) denote material-specific
diffusion coefficients, which need to be determined experimentally.
Finally, the internal heat sources
\begin{equation}
Q = L_v \bigl( D_w( w, \theta ) \nabla w + D_\varphi( w, \theta )
\nabla \varphi - \frac{\partial}{\partial t} \left[ (e - w) \varphi
\rho_{p,s}( \theta ) \right] \bigr)
\end{equation}
quantify the influence of phase changes in pores by means of the
latent heat of evaporation of water $L_v$~(Jkg$^{-1}$).

To close the model, Kiessl in~\cite{Kiessl1983} related the
auxiliary variables $w$ and $\varphi$ to a dimensionless moisture
potential $\Phi = u^2$ via monotone, material-dependent,
transformations
\begin{eqnarray}
w = f( \Phi ), && \varphi = g( \Phi ),
\end{eqnarray}
satisfying $f(0) = g(0) = 0$ and $\frac{\de g}{\de \Phi}(0)=1$. In
particular, $f$ denotes the sorption isotherm, whereas $g$ reflects
the pore size distribution. By employing these identifies, the
individual coefficients in~\eqref{eq01} receive the
form~(cf.~\cite{DalikDanecekStastnik1996})
\begin{subequations}
\begin{eqnarray}
B^1 & = & \rho_0 c_0 \theta + \rho_w c_w g( \Phi ) \theta + L_v (e -
f( \Phi )) g( \Phi ) \rho_{p,s}( \theta ),
\\
B^2 & = & \rho_w f( \Phi) + (e - f( \Phi )) g( \Phi ) \rho_{p,s}(
\theta ),
\\
a^{11} & = & \lambda( f( \Phi ), \theta ),
\\
a^{12} & = & L_v \left( D_w( f( \Phi ), \theta ) \frac{\de f( \Phi
)}{\de \Phi} + D_\varphi( f( \Phi ), \theta ) \frac{\de g( \Phi
)}{\de \Phi} \right),
\\
a^{21} & = & D_\theta( f( \Phi ), \theta ),
\\
a^{22} & = & D_w( f( \Phi ), \theta ) \frac{\de f( \Phi )}{\de \Phi}
+ D_\varphi( f( \Phi ), \theta ) \frac{\de g( \Phi )}{\de \Phi}.
\end{eqnarray}
\end{subequations}

\subsection{The K\"{u}nzel model}\label{sec:applications_kunzel}
%%%%%%%%%%%%%%%%%%%%%%%%%%%%%%%%%
In the K\"{u}nzel framework, the heat balance is described using
identical expressions for the enthalpy~\eqref{eq:kiessl_enthalpy}
and the heat flux~\eqref{eq:kiessl_heat_flux} as previously. In
addition, the moisture density is simplified into
\begin{equation}
M = \rho_w w
\end{equation}
and the moisture flux attains a form
\begin{equation}
\bfj_m = - \bigl( \hat{D}_\varphi (\varphi, \theta) \nabla \varphi +
\frac{\delta( \theta )}{\mu} \nabla \left( \varphi p_s ( \theta
)\right) \bigr),
\end{equation}
in which $\hat{D}_\varphi$~(kgm$^{-1}$s$^{-1}$) stands for the
liquid conduction coefficient,
$\delta$~(kgm$^{-1}$s$^{-1}$Pa$^{-1}$) is the vapor diffusion
coefficient in air, $\mu$~(-) is the vapor diffusion resistance
factor of a porous material and $p_s$~(Pa) is the vapor saturation
pressure. This yields the heat source term given by
\begin{equation}
Q = L_v \nabla \cdot \left( \frac{\delta( \theta )}{\mu} \nabla
\left( \varphi p_s ( \theta )\right) \right).
\end{equation}

The relative humidity is chosen as the second unknown, $u^2 =
\varphi$, and is used to express the associated volumetric moisture
content in the form
\begin{equation}
w = h ( \varphi ),
\end{equation}
where $h$ is a monotone moisture storage function with $h(0)=0$.
Altogether, the coefficients in system~\eqref{eq01} read as %%
\begin{subequations}
\begin{eqnarray}
B^1 & = & \rho_0 c_0 \theta + \rho_w c_w h( \varphi ) \theta,
\\
B^2 & = & \rho_w h( \varphi ),
\\
a^{11} & = & \lambda( h( \varphi ), \theta ) + L_v
\frac{\delta(\theta)}{\mu} \varphi \frac{\de p_s( \theta )}{\de
\theta},
\\
a^{12} & = & L_v \frac{\delta(\theta)}{\mu} p_s( \theta ),
\\
a^{21} & = & \frac{\delta( \theta )}{\mu} \varphi \frac{p_s( \theta
)}{\de \theta},
\\
a^{22} & = & D_\varphi( \varphi, \theta ) +
\frac{\delta(\theta)}{\mu} p_s( \theta ).
\end{eqnarray}
\end{subequations}

\subsection{Structure conditions (A1) and (A2)}
%%%%%%%%%%%%%%%%%%%%%%%%%%%%%%%%%%%%%%%%%%%%%%%%
\JZtext{The structure conditions (A1)--(A2) closely reflect the physical
constraints on the underlying transport models and experimental observations.
Concretely, the model parameters (such as e.g. $f(\Phi)$, $g(\Phi)$ and
$\rho_{p,s}(\theta)$ in the Kiessl model, or $h(\varphi)$ and $p_s(\theta)$
in the K\"{u}nzel model) are obtained by fitting smooth
functions to experimental data, determined for a limited range of state variables.  The required regularity and boundedness of
coefficients $\bfB$, $a^{jk}$ and positivity of $a^{jk}$ is therefore ensured.
The increasing character of $\bfB$ is consistent with the fact that both the
specific enthalpy $H$ and the moisture density $M$ increase with an increasing
temperature and the moisture-related variable, respectively. The ellipticity
condition~\eqref{gen_par2} is satisfied due to the fact that the Soret- and
Dufour-type fluxes, quantified by $a^{12}$ and $a^{21}$, are dominated by the
diagonal contributions $a^{11}$ and $a^{22}$, see
also~\cite{DalikDanecekVala2000}. Therefore, any physically correct form of
$a^{jk}$ must meet this condition. The validity of the
assumption~\eqref{gen_par1} then follows from the same physical reasoning.}

%--------------------------------------------------------------------------------

%--------------------------------------------------------------------------------

\appendix

%--------------------------------------------------------------------------------
\section{Proof of Lemma \ref{lm_sup_10}}\label{proof_1}
Discretize \eqref{eq_lin01} in time and replace $\bfu_{\ell}'(t_n)$
by the backward difference quotient
$\partial_t^{-h}(\bfw_{\ell})_n=[(\bfw_{\ell})_n-(\bfw_{\ell})_{n-1}]/h$,
where $h>0$ is a time step. Suppose $r=T/h$ is an integer. For
simplicity, let us write $\bfw_{\ell}=(\bfw_{\ell})_n$,
$\bff_{\ell}=(\bff_{\ell})_n$. We have to solve, successively for
$n=1,\cdots,r$, the steady problems
\begin{multline}\label{app_rothe_01}
\sum_{\ell=1}^{M} \int_{\Omega_{\ell}}
\beta^{ji}_{\ell}\partial_t^{-h}w^i_{\ell} \, v^j \, {\rm d}\bfx +
\sum_{\ell=1}^{M} \int_{\Omega_{\ell}}
  \kappa^{ji}_{\ell}\nabla w^i_{\ell} \cdot \nabla v^j  \,
{\rm d}\bfx   \\
+\sum_{\ell=1}^{M}\int_{\partial\Omega_{\ell}\cap\Gamma}
\nu^j_{\ell} \; w^j_{\ell}\,v^j \, {\rm d}S =
\sum_{\ell=1}^{M}\int_{\Omega_{\ell}} f^j_{\ell} \, v^j \, {\rm
d}\bfx
\end{multline}
which hold for every $\bfv\in \mathbf{W}^{1,2}$ and
$(\bfw_{\ell})_0={\bf0} \mbox{ in } \Omega_{\ell}$. Test
\eqref{app_rothe_01} by
$[v^1,v^2]=[\kappa^{21}_{\ell}\varphi^1,\kappa^{12}_{\ell}\varphi^2]$
and define the bilinear form $\mathfrak{A}(\bfw,\bfvarphi)$;
\begin{multline}
\mathfrak{A}(\bfw_{\ell},\bfvarphi) =\frac{1}{h}\sum_{\ell=1}^{M}
\int_{\Omega_{\ell}} \kappa^{pj}_{\ell}\beta^{ji}_{\ell}w^i_{\ell}
\, \varphi^j \, {\rm d}\bfx + \sum_{\ell=1}^{M} \int_{\Omega_{\ell}}
  \kappa^{pj}_{\ell}\kappa^{ji}_{\ell}\nabla w^i_{\ell}\cdot \nabla \varphi^j  \,
{\rm d}\bfx
\nonumber \\
 + \sum_{\ell=1}^{M}\int_{\partial\Omega_{\ell}\cap\Gamma}
\kappa^{pj}_{\ell}\nu^j_{\ell} \; w^j_{\ell}\, \varphi^j \, {\rm
d}S ,\quad p=1,2,\; p\neq j,
\end{multline}
for every $\bfvarphi\in \mathbf{W}^{1,2}$. Set
$\bfvarphi=\bfw_{\ell}$. Now, \eqref{lingen_par1},
\eqref{lingen_par2} and the Friedrichs inequality yield the
$\mathbf{W}^{1,2}$-ellipticity, i.e. there exists $c>0$ such that
\begin{equation}
c\|\bfw_{\ell}\|^2_{\mathbf{W}^{1,2}_{\ell}}\leq
|\mathfrak{A}(\bfw_{\ell},\bfw_{\ell})|
\end{equation}
for all $\bfw_{\ell}\in \mathbf{W}_{\ell}^{1,2}$. Using the
H\"{o}lder inequality and the standard trace theorem one obtains the
continuity of $\mathfrak{A}$, i.e. the inequality
\begin{equation}
|\mathfrak{A}(\bfw_{\ell},\bfz_{\ell})| \leq c
\|\bfw_{\ell}\|_{\mathbf{W}^{1,2}_{\ell}}\|\bfz_{\ell}\|_{\mathbf{W}^{1,2}_{\ell}}
\end{equation}
which holds for all $\bfw_{\ell}, \bfz_{\ell}\in
\mathbf{W}_{\ell}^{1,2}$ and for some positive constant $c$. The
linearity of $\mathfrak{A}:\mathbf{W}^{1,2}_{\ell}\rightarrow
(\mathbf{W}^{1,2}_{\ell})^*$ is obvious. Hence, for every
$\bff_{\ell}\in\mathbf{L}_{\ell}^{2}\subset(\mathbf{W}^{1,2}_{\ell})^*$
the Lax-Milgram theorem yields the existence of the weak solution
$\bfw_{\ell} \in \mathbf{W}_{\ell}^{1,2}$. To get higher regularity
results (with respect to time), define $\bfw_{\ell} \in
\mathbf{W}^{1,2}_{\ell}$ by \eqref{app_rothe_01} and test
\eqref{app_rothe_01} by
$[v^1,v^2]=[\kappa^{21}_{\ell}\partial_t^{-h}w^1_{\ell}
,\kappa^{12}_{\ell}\partial_t^{-h}w^2_{\ell} ]$ to get
\begin{multline}\label{rothe_05}
\sum_{\ell=1}^{M} \int_{\Omega_{\ell}}
\kappa^{pj}_{\ell}\beta^{ji}_{\ell}\partial_t^{-h}w^i_{\ell} \,
\partial_t^{-h}w^j_{\ell} \, {\rm d}\bfx
\\
+ \sum_{\ell=1}^{M} \int_{\Omega_{\ell}}
\kappa^{pj}_{\ell}\kappa^{ji}_{\ell}\nabla w^i_{\ell} \nabla
(\partial_t^{-h}w^j_{\ell})  \, {\rm d}\bfx
 + \sum_{\ell=1}^{M}\int_{\partial\Omega_{\ell}\cap\Gamma}
\kappa^{pj}_{\ell}\nu^j_{\ell} \;
w^j_{\ell}\,\partial_t^{-h}w^j_{\ell} \, {\rm d}S
\\
= \sum_{\ell=1}^{M}\int_{\Omega_{\ell}} \kappa^{pj}_{\ell}
f^j_{\ell} \,
\partial_t^{-h}w^j_{\ell} \, {\rm d}\bfx,\quad p=1,2,\; p\neq j.
\end{multline}
Denote by
\begin{multline}
\Phi_{\ell}\left[\nabla(\bfw_{\ell})_n\right]=
\kappa^{21}_{\ell}\kappa^{11}_{\ell}\frac{1}{2}\big|\nabla
(w^1_{\ell})_n\big|^2
+\kappa^{12}_{\ell}\kappa^{22}_{\ell}\frac{1}{2}\big|\nabla
(w^2_{\ell})_n\big|^2
\\
+\kappa^{12}_{\ell}\kappa^{21}_{\ell}\nabla
(w^1_{\ell})_n\cdot\nabla (w^2_{\ell})_n, \quad n=1,\dots,r.
\end{multline}
Now we can estimate
\begin{displaymath}
\Phi'_{\ell}[(\nabla\bfw_{\ell})_n]\cdot((\nabla\bfw_{\ell})_n-(\nabla\bfw_{\ell})_{n-1})\geq
\Phi_{\ell}[(\nabla\bfw_{\ell})_n]-\Phi_{\ell}[(\nabla\bfw_{\ell})_{n-1}]
\end{displaymath}
because $\Phi_{\ell}$ is convex. Thus, using Young's inequality, one
obtains
\begin{multline}\label{rothe_06}
\sum_{\ell=1}^{M} \int_{\Omega_{\ell}}
\kappa^{pj}_{\ell}\beta^{ji}_{\ell}\partial_t^{-h}w^i_{\ell} \,
\partial_t^{-h}w^j_{\ell} \, {\rm d}\bfx + \frac{1}{h} \sum_{\ell=1}^{M}
\int_{\Omega_{\ell}}
\Phi_{\ell}[(\nabla\bfw_{\ell})_n]-\Phi_{\ell}[(\nabla\bfw_{\ell})_{n-1}]
\, {\rm d}\bfx
\\
 + \frac{1}{h} \sum_{\ell=1}^{M}
 \int_{\partial\Omega_{\ell}\cap\Gamma}
\kappa^{pj}_{\ell}\nu^j_{\ell} \left( \frac{1}{2}
\big|(w^j_{\ell})_n\big|^2-\frac{1}{2}
\big|(w^j_{\ell})_{n-1}\big|^2\right)\,{\rm d}S
\\
\leq  C(\epsilon) \sum_{\ell=1}^{M}
\|\bff_{\ell}\|^2_{\mathbf{L}^2_{\ell}} + \epsilon \, c_2
\sum_{\ell=1}^{M} \int_{\Omega_{\ell}}
\|\partial_t^{-h}\bfw_{\ell}\|^2_{\mathbf{L}^2_{\ell}} \, {\rm
d}\bfx,\quad p=1,2,\; p\neq j,
\end{multline}
with some arbitrarily small constant $\epsilon$. Note that
\eqref{lingen_par1} yields the estimate of the parabolic term
\begin{displaymath}
c \, \sum_{\ell=1}^{M} \int_{\Omega_{\ell}}
\|\partial_t^{-h}\bfw_{\ell}\|^2_{\mathbf{L}^2_{\ell}} \, {\rm
d}\bfx \leq \sum_{\ell=1}^{M} \int_{\Omega_{\ell}}
\kappa^{pj}_{\ell}\beta^{ji}_{\ell}\partial_t^{-h}w^i_{\ell} \,
\partial_t^{-h}w^j_{\ell} \, {\rm d}\bfx,\quad p=1,2,\; p\neq j,
\end{displaymath}
where $c$ depends on $\beta^{ji}_{\ell}$ and $\kappa^{ji}_{\ell}$.
Based on the estimate \eqref{rothe_06}, the same way as in
\cite[Proof of Theorem 8.16]{Roubicek2005} we can prove the
existence of the solution $\bfu_{\ell}\in
L^{\infty}(0,T;\;\mathbf{W}^{1,2}_{\ell})\hookrightarrow
L^{\infty}(0,T;\;\mathbf{L}^{2}_{\ell})$ with $\bfu'_{\ell}(t)\in
L^2(0,T;\;\mathbf{L}^{2}_{\ell})$ and the estimate
\begin{equation}\label{estimate_a0}
\|\bfu'_{\ell}(t)\|_{L^{2}(0,T;\mathbf{L}^2_{\ell})}
+\|\bfu_{\ell}\|_{L^{\infty}(0,T;\mathbf{L}^2_{\ell})} \leq c \|
\bff_{\ell}\|_{L^{2}(0,T;\mathbf{L}^2_{\ell})}.
\end{equation}
Now we can proceed as in \cite{HungAnh2008}. Rewrite the system
\eqref{eqlin01}--\eqref{eqlin05} in the form
\begin{equation}\label{linear_transmission_strong_2}
\left\{
\begin{array}{rclll}
-\nabla \cdot \left( \kappa^{ji}_{\ell}\nabla u^i_{\ell} \right) &=&
F^j_{\ell} := f^j_{\ell} - \beta^{ji}_{\ell}\frac{\partial
u^i_{\ell}}{\partial t}&  {\rm in }
& Q_{\ell T},\\
\kappa^{ji}_{\ell}\frac{\partial u^i_{\ell}}{\partial
\bfn_{\ell}}+\nu^j_{\ell} u^j_{\ell}&=&0 & {\rm on}
&S_{\ell T},\\
u^j_{\ell} &=& u^j_{m} & {\rm on } & \Gamma_{m\ell}\times(0,T),\\
\kappa^{ji}_{\ell}\frac{\partial u^i_{\ell}}{\partial
\bfn_{\ell}}+\kappa^{ji}_{m}\frac{\partial u^i_{m}}{\partial
\bfn_{m}}&=&0 & {\rm on} & \Gamma_{m\ell}\times(0,T).
\end{array}\right.
\end{equation}
Since $\bfu'(t)_{\ell}\in L^2(0,T;\mathbf{L}^{2}_{\ell})$ we have
$F^j_{\ell}(\bfx,t)\in \mathbf{L}^{2}_{\ell}$ for a.e. $t\in(0,T)$.
According to results for stationary transmission problem (see
\ref{regularity_stationary}, Corollary \ref{corollary_appendix}) we
have $\bfu_{\ell}(t)\in \mathbf{W}^{2,2}_{\ell}$ and the estimate
\begin{equation}\label{estimate_a1}
\|\bfu_{\ell}(t)\|_{\mathbf{W}^{2,2}_{\ell}} \leq c
\|\bff_{\ell}(t)\|_{\mathbf{L}_{\ell}^2}
\end{equation}
holds for almost every $t\in(0,T)$, where the constant $c$ is
independent of $t$. Raising \eqref{estimate_a1} and integrating both
sides in the inequality \eqref{estimate_a1} with respect to time, we
get $\bfu_{\ell}\in L^2(0,T;\mathbf{W}^{2,2}_{\ell})$ and taking
into account the estimate \eqref{estimate_a0} we arrive at
\begin{equation}
\|\bfu'_{\ell}(t)\|_{L^{2}(0,T;\mathbf{L}^2_{\ell})}
+\|\bfu_{\ell}\|_{L^{2}(0,T;\mathbf{W}^{2,2}_{\ell})}
+\|\bfu_{\ell}\|_{L^{\infty}(0,T;\mathbf{L}^2_{\ell})} \leq c
\|\bff_{\ell}\|_{L^{2}(0,T;\mathbf{L}^2_{\ell})}.
\end{equation}
Now taking the derivative of \eqref{eqlin01}--\eqref{eqlin05} with
respect to time, considering $\bff_{\ell}\in \mathcal{Y}_{\ell,T}$
(including the compatibility condition on $\bff_{\ell}(\bfx,0))$, we
conclude $\bfu''_{\ell}(t)\in L^{2}(0,T;\mathbf{L}^2_{\ell})$,
$\bfu'_{\ell}(t)\in
L^{2}(0,T;\mathbf{W}^{2,2}_{\ell})\cap
L^{\infty}(0,T;\mathbf{L}^2_{\ell})$
and the estimate \eqref{eq11b} follows. The linearity of Problem
$(P_f)$ and the estimate \eqref{eq11b} yield the uniqueness.

%--------------------------------------------------------------------------------

\section{Proof of Lemma \ref{lm_sup_11}}\label{proof_2}
Let $\bfg_{\ell}\in
{W^{2,2}(0,T;(\mathbf{W}^{1/2,2}_{\partial\Omega_{\ell}\cap\Gamma})^*)}$.
We are looking for the solution of the problem defined via
\begin{multline}\label{eq_app_b}
\sum_{\ell=1}^{M} \int_{\Omega_{\ell}}
\beta^{ji}_{\ell}\frac{\partial \xi^i_{\ell}}{\partial t} \, v^j \,
{\rm d}\bfx + \sum_{\ell=1}^{M} \int_{\Omega_{\ell}}
  \kappa^{ji}_{\ell}\nabla \xi^i_{\ell} \cdot \nabla v^j  \,
{\rm d}\bfx  \\
 + \sum_{\ell=1}^{M}\int_{\partial\Omega_{\ell}\cap\Gamma} \nu^j_{\ell}
\; \xi^j_{\ell}\,v^j \, {\rm d}S = \sum_{\ell=1}^{M}\langle
\bfG''_{\ell}(t) ; \bfv
\rangle_{(\mathbf{W}^{1,2}_{\ell})^*,\mathbf{W}^{1,2}_{\ell}}
\end{multline}
to be satisfied for every $\bfv\in\mathbf{W}^{1,2}$, almost every
$t\in (0,T)$ and $\bfxi_{\ell}(\bfx,0)={\bf0}$ in $\Omega_{\ell}$.
The duality of $\langle \bfG_{\ell} ; \bfv
\rangle_{(\mathbf{W}^{1,2}_{\ell})^*,\mathbf{W}^{1,2}_{\ell}}$
corresponds to
\begin{equation}
\langle \bfG_{\ell} ; \bfv
\rangle_{(\mathbf{W}^{1,2}_{\ell})^*,\mathbf{W}^{1,2}_{\ell}} =
\sum_{\ell=1}^{M} \int_{\partial\Omega_{\ell}\cap\Gamma}
g^j_{\ell}\, v^j \; {\rm d}S .
\end{equation}
Approximate \eqref{eq_app_b} in time by discretization and replace
$\bfxi_{\ell}'(t_n)$ by the backward difference quotient
$\partial_t^{-h}(\bfw_{\ell})_n=[(\bfw_{\ell})_n-(\bfw_{\ell})_{n-1}]/h$,
where $h>0$ is a time step. Suppose $r=T/h$ is an integer. Let us
write $\bfw_{\ell}=(\bfw_{\ell})_n$ and test \eqref{eq_app_b} by
$[v^1,v^2]=[\kappa^{21}_{\ell}\varphi^1,\kappa^{12}_{\ell}\varphi^2]$.
We have to solve, successively for $n=1,\cdots,r$, the steady
problems
\begin{multline}\label{app_rothe_b1}
\sum_{\ell=1}^{M} \int_{\Omega_{\ell}}\kappa^{pj}_{\ell}
\beta^{ji}_{\ell}\partial_t^{-h}w^i_{\ell} \, \varphi^j \, {\rm
d}\bfx \\
+ \sum_{\ell=1}^{M}  \int_{\Omega_{\ell}}
 \kappa^{pj}_{\ell} \kappa^{ji}_{\ell}\nabla w^i_{\ell} \nabla \varphi^j  \,
{\rm d}\bfx + \sum_{\ell=1}^{M}
\int_{\partial\Omega_{\ell}\cap\Gamma}
\kappa^{pj}_{\ell}\nu^j_{\ell} \; w^j_{\ell}\,\varphi^j \, {\rm
d}S
\\
=\sum_{\ell=1}^{M} \kappa^{pj}_{\ell}\langle (G''(t_n))^j_{\ell} ;
\varphi^j \rangle_{({W}^{1,2}_{\ell})^*,{W}^{1,2}_{\ell}},\quad
p=1,2,\; p\neq j,
\end{multline}
to hold for every $\bfvarphi\in \mathbf{W}^{1,2}$ and
$(\bfw_{\ell})_0={\bf0}$. Test \eqref{app_rothe_b1} by
$\bfvarphi=(\bfw_{\ell})_n$ to get the estimate
\begin{multline}\label{app_rothe_b1b}
c_1\sum_{\ell=1}^{M}
\left(\frac{1}{2}\|(\bfw_{\ell})_n\|_{\mathbf{L}^2}
-\frac{1}{2}\|(\bfw_{\ell})_0\|_{\mathbf{L}^2}\right)
\\
+ c_2 h \left( \sum_{\ell=1}^{M}\sum_{m=1}^{n}\left(
\int_{\Omega_{\ell}} |\nabla (\bfw_{\ell})_m|^2 {\rm d}\bfx
+\int_{\partial\Omega_{\ell}\cap\Gamma} | (\bfw_{\ell})_m|^2 {\rm
d}S \right) \right)\\
\leq  \sum_{\ell=1}^{M}\sum_{m=1}^{n} \kappa^{pj}_{\ell}\langle
(G''(t_n))^j_{\ell} ; (w^j_{\ell})_n
\rangle_{({W}^{1,2}_{\ell})^*,{W}^{1,2}_{\ell}},\quad p=1,2,\; p\neq j.
\end{multline}
Now we can proceed as in \cite[Proof of Lemma 8.6, Proof of Theorem
8.9]{Roubicek2005} to prove the existence of the weak solution
$\bfxi_{\ell}\in L^2(0,T;\mathbf{W}^{1,2}_{\ell})$ (as the limit of
Rothe sequences) with the estimate
\begin{equation}\label{estimate_b1}
\|\bfxi_{\ell}\|_{L^2(0,T;\mathbf{W}^{1,2}_{\ell})} \leq c
\|\bfG''_{\ell}(t)\|_{L^2(0,T;(\mathbf{W}^{1,2}_{\ell})^*)}.
\end{equation}
Let us note that $\bfxi_{\ell}$ stands for $\bfu''_{\ell}(t)$. Hence
$\bfu''_{\ell}(t)\in L^2(0,T;\mathbf{W}^{1,2}_{\ell})\hookrightarrow
L^2(0,T;\mathbf{L}^{2}_{\ell})$. Further, let
$\bfg_{\ell}\in{W^{1,2}(0,T;\mathbf{W}^{1/2,2}_{\partial\Omega_{\ell}\cap\Gamma})}$.
The standard theory for parabolic problems yields
$\bfu'_{\ell}(t)\in L^{\infty}(0,T;\mathbf{L}^{2}_{\ell})$. Using
the same procedure as in \ref{proof_1} and according to results for
stationary transmission problem (see \ref{regularity_stationary},
Corollary \ref{corollary_appendix}) we conclude $\bfu'_{\ell}(t) \in
L^2(0,T;\mathbf{W}^{2,2}_{\ell}) \cap
L^{\infty}(0,T;\mathbf{L}^{2}_{\ell}) $ and (combining with
\eqref{estimate_b1})
\begin{multline}
\|{ \bfu''_{\ell}(t)}\|_{L^{2}(0,T;\mathbf{L}^2_{\ell})}
+\|{ \bfu'_{\ell}(t)}\|_{L^{2}(0,T;\mathbf{W}^{2,2}_{\ell})}
+\|{ \bfu'_{\ell}(t)}\|_{L^{\infty}(0,T;\mathbf{L}^2_{\ell})}
\\
\leq c  \left( \|
\bfg_{\ell}\|_{W^{2,2}(0,T;(\mathbf{W}^{1/2,2}_{\partial\Omega_{\ell}\cap\Gamma})^*)}
+ \|
\bfg_{\ell}\|_{W^{1,2}(0,T;\mathbf{W}^{1/2,2}_{\partial\Omega_{\ell}\cap\Gamma})}
\right).
\end{multline}
The linearity of Problem $(P_g)$ and the estimate
\eqref{estimate_b1} yield the uniqueness.

%--------------------------------------------------------------------------------

\section{Transmission problem for elliptic systems in a multi-layer
structure}\label{regularity_stationary} The boundary transmission
problem for the elliptic system for the subdomain $\Omega_{\ell}$ is
formulated, in the expanded form, as
\begin{equation}\label{transm_elliptic_system}
\left\{
\begin{array}{rclll}
-\nabla\cdot\left(\varepsilon^{11}_{\ell}(\bfx)\nabla
u^1_{\ell}\right)
-\nabla\cdot\left(\varepsilon^{12}_{\ell}(\bfx)\nabla
u^2_{\ell}\right)&=& f^1_{\ell} &  {\rm in } & \Omega_{\ell},
\\
-\nabla\cdot\left(\varepsilon^{21}_{\ell}(\bfx)\nabla
u^2_{\ell}\right)
-\nabla\cdot\left(\varepsilon^{22}_{\ell}(\bfx)\nabla
u^2_{\ell}\right)&=& f^2_{\ell} &  {\rm in } & \Omega_{\ell},
\\
\varepsilon^{11}_{\ell}(\bfx)\frac{\partial u^1_{\ell}}{\partial
\bfn_{\ell}} + \varepsilon^{12}_{\ell}(\bfx)\frac{\partial
u^2_{\ell}}{\partial \bfn_{\ell}} +\alpha^{1}_{\ell}u^1_{\ell}
&=&g^1_{\ell} & {\rm on} &
\partial\Omega_{\ell}\cap\Gamma,
\\
\varepsilon^{21}_{\ell}(\bfx)\frac{\partial u^1_{\ell}}{\partial
\bfn_{\ell}} +\varepsilon^{22}_{\ell}(\bfx)\frac{\partial
u^2_{\ell}}{\partial \bfn_{\ell}} +\alpha^{2}_{\ell}u^2_{\ell}
&=&g^2_{\ell} & {\rm on} &
\partial\Omega_{\ell}\cap\Gamma,
\\
u^1_{\ell} &=& u^1_{m} & {\rm on } & \Gamma_{m\ell},
\\
u^2_{\ell} &=& u^2_{m} & {\rm on } & \Gamma_{m\ell},
\\
\varepsilon^{11}_{\ell}(\bfx)\frac{\partial u^1_{\ell}}{\partial
\bfn_{\ell}} + \varepsilon^{12}_{\ell}(\bfx)\frac{\partial
u^2_{\ell}}{\partial \bfn_{\ell}} +
\varepsilon^{11}_{m}(\bfx)\frac{\partial u^1_{m}}{\partial \bfn_{m}}
+ \varepsilon^{12}_{m}(\bfx)\frac{\partial u^2_{m}}{\partial
\bfn_{m}}&=&0 & {\rm on} & \Gamma_{m\ell},
\\
\varepsilon^{21}_{\ell}(\bfx)\frac{\partial u^2_{\ell}}{\partial
\bfn_{\ell}} + \varepsilon^{22}_{\ell}(\bfx)\frac{\partial
u^2_{\ell}}{\partial \bfn_{\ell}} +
\varepsilon^{21}_{m}(\bfx)\frac{\partial u^1_{m}}{\partial \bfn_{m}}
+ \varepsilon^{22}_{m}(\bfx)\frac{\partial u^2_{m}}{\partial
\bfn_{m}}&=&0 & {\rm on} & \Gamma_{m\ell}.
\end{array}\right.
\end{equation}
Here we assume that the problem \eqref{transm_elliptic_system} is
elliptic and has a unique weak solution $\bfu_{\ell} \in
\mathbf{W}^{1,2}_{\ell}$  for $\bff_{\ell} \in \mathbf{L}_{\ell}^2$
and  $\bfg_{\ell} \in \mathbf{W}_{\ell,\Gamma}^{1/2,2}$. Further we
consider that $\varepsilon^{ji}_{\ell}(\bfx)$ are positive Lipschitz
continuous functions and $\alpha^{j}_{\ell}$ are prescribed
constants.

Elliptic boundary value problems in cornered plane domains are
extensively investigated in the literature, see e.g.
\cite{Dauge1988,Kon1967,KozMazRoss1997,KozMazRoss2001,KufSan1987}.
The behavior of local solutions of general linear and semilinear
transmission problems is studied in
\cite{NiSa1994a,NiSa1994b,Sandig2000}. We adapt the general
framework stated in the literature to calculate the regularity of
the plane transmission problem for the elliptic system of equations
\eqref{transm_elliptic_system}.

It is known (cf. \cite{KufSan1987,Sandig2000}) that, in general, the
boundary singularities may occur near corner points at the boundary
$\partial \Omega$, the points at the boundary where the boundary
conditions change their type, the crossing points of interfaces,
corner points of inclusions or points, where the interfaces
$\Gamma_{m\ell}$ intersect the exterior boundary of the domain
$\Omega$. Taking into account the assumptions on admissible domains
introduced in Subsection \ref{Admissible domains}, only the points
where the interfaces $\Gamma_{m\ell}$ intersect the exterior
boundary are of importance in our analysis of vertex singularities.
Hence, let $\mathcal{M}$ be the set of all boundary points
$A\in\Gamma\cap\Gamma_{m\ell}$, $m$, $\ell=1,\dots,M$. As well
known, the local regularity is valid outside an arbitrarily small
neighborhood of the points $A\in\mathcal{M}$. Hence it suffices to
prove the regularity for the solution $\bfu_{\ell}$ with small
supports. For solutions with arbitrary support the assertion then
can be easily proved by means of a partition of unity on $\Omega$.
Let $A$ be an arbitrary point from the set $\mathcal{M}$ and let the
support of $\bfu_{\ell}$ be contained in a sufficiently small
neighborhood $\mathcal{{U}}(A)$ of the point $A$. Let $D$ be a
diffeomorphic mapping $\Omega \cap \mathcal{{U}}(A)$ onto
$\mathcal{K}_A\cap C_A$, where $\mathcal{K}_A$ is an angle with
vertex at the origin (shifted into $A$) and $C_A$ is a unit circle
centered at the origin. Considering a zero-extension outside $C_A$,
a new boundary value problem with ``frozen coefficients''
$\varepsilon^{ji}_{\ell}({\bf0})$ is defined in an infinite angle
$\mathcal{K}_A=\mathcal{K}_{\ell}\cup\mathcal{K}_{\ell+1}$ and the
problem of regularity of the solution coincides with the original
problem near the corner point $A$. One applies certain regularity
theorem for the boundary value problem in an infinite angle
$\mathcal{K}_A$. Hence, following \cite[Chapter 3 and 4]{NiSa1994a},
we localize the boundary value problem
\eqref{transm_elliptic_system} (see Fig. \ref{localization}), i.e.
identify the origin ${\bf0}$ of coordinates with the point $A$,
``freeze'' the coefficients $\varepsilon^{ji}_{\ell}({\bf0})$ and
multiply the weak solution by a cut-off function $\eta(|\bfx|) \in
{C}^{\infty}(\mathbb{R}^2)$, $0 \leq \eta(|\bfx|)$, given by
\begin{displaymath}
\eta(|\bfx|) =  \quad  \left\{
\begin{array}{ccl}
1 & {\rm for} &  |\bfx| < \epsilon / 2 ,
\\
0 & {\rm for} &  |\bfx| > \epsilon.
\end{array} \right.
\end{displaymath}
\begin{figure}[h]
\centering
\includegraphics[width=7.0cm]{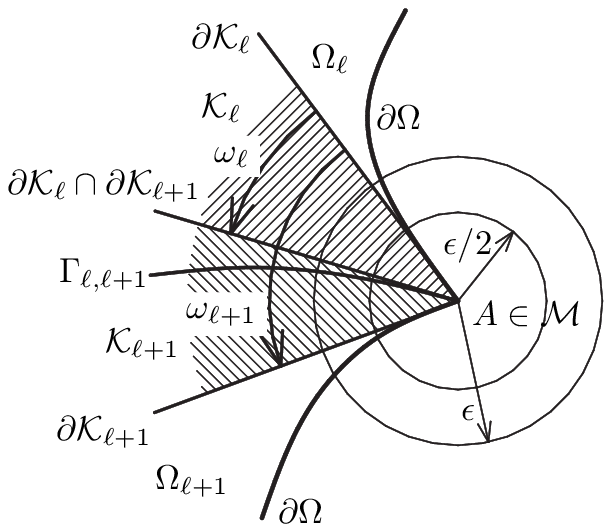}
\caption{Localization principle.}\label{localization}
\end{figure}
Denote $\bfw_{m} = \eta\bfu_{m}$, $m=\ell,\ell+1$, and consider the
following problem with frozen coefficients (now written in compact
form using the Einstein summation convention for indices $i$ and $j$
running from 1 to 2)
\begin{equation}\label{transmission_problem_angle}
\left\{
\begin{array}{rclll}
- \varepsilon^{ji}_{m}({\bf0})\Delta w^i_{m}  &=& F^j_{m} & {\rm in
} & \mathcal{K}_{m},
\\
\varepsilon^{ji}_{m}({\bf0})\frac{\partial w^i_{m}}{\partial
\bfn_{m}}&=&G^j_{m} & {\rm on} &
\partial\mathcal{K}_{m}\cap\Gamma,
\\
w^j_{\ell} &=& w^j_{\ell+1} & {\rm on } & \partial\mathcal{K}_{\ell}
\cap\partial\mathcal{K}_{\ell+1},
\\
\varepsilon^{ji}_{\ell}({\bf0})\frac{\partial w^i_{\ell}}{\partial
\bfn_{\ell}}+\varepsilon^{ji}_{\ell+1}({\bf0})\frac{\partial
w^i_{\ell+1}}{\partial \bfn_{\ell+1}}&=&0 & {\rm on} &
\partial\mathcal{K}_{\ell} \cap\partial\mathcal{K}_{\ell+1},
\end{array}\right.
\end{equation}
where
\begin{eqnarray}\label{cut_right_hand}
F^{j}_{m}&=&-u^j_{m}\Delta
\eta-2\varepsilon^{ji}_{m}({\bf0})\frac{\partial u^i_{m}}{\partial
x_k}\frac{\partial \eta}{\partial x_k}+f^j_{m}\eta,    \nonumber
\\
G^{j}_{m} &=& g^j_{m}\eta-\alpha^j_{m}u^j_{m}\eta . \nonumber
\end{eqnarray}
The regularity of $\bfu_{m}$ in a neighborhood of the point $A$ is
determined by the smoothness of $\bfw_{m}$ near $A$.
Using polar coordinates $(r,\omega)$ in
\eqref{transmission_problem_angle} we arrive at the system
\begin{equation}\label{transmission_problem_polar}
\left\{
\begin{array}{rcllll}
- \varepsilon^{ji}_{m}({\bf0})\left( \frac{\partial ^2
\overline{w}^i_{m}}{\partial r^2} + \frac{1}{r} \frac{\partial
\overline{w}^i_{m}}{\partial r} + \frac {1}{r^2} \frac {\partial ^2
\overline{w}^i_{m}}{\partial \omega ^2} \right) &=&
\overline{F}^j_{m} \; {\rm in } \; \overline{S}_{m}, \;
m=\ell,\ell+1,% &&&
\\
 \varepsilon^{ji}_{\ell}({\bf0})\frac{\partial
 \overline{w}^i_{\ell}}{\partial \omega}(r,0)&=&
\overline{G}^j_{\ell}(r,0), &&&
\\
 \varepsilon^{ji}_{\ell+1}({\bf0})\frac{\partial
 \overline{w}^i_{\ell+1}}{\partial
\omega}(r,\omega_{\ell+1})&=&
\overline{G}^j_{\ell+1}(r,\omega_{\ell+1}), &&&
\\
\overline{w}^j_{\ell}(r,\omega_{\ell})  &=&
\overline{w}^j_{\ell+1}(r,\omega_{\ell}), &&
\\
 \varepsilon^{ji}_{\ell}({\bf0})\frac{\partial \overline{w}^i_{\ell}}{\partial
\omega}(r,\omega_{\ell}) &=&
\varepsilon^{ji}_{\ell+1}({\bf0})\frac{\partial
\overline{w}^i_{\ell+1}}{\partial \omega}(r,\omega_{\ell}), &&
\end{array}\right.
\end{equation}
where $\overline{S}_{\ell}$ and $\overline{S}_{\ell+1}$,
respectively, is an infinite half-strip
\begin{eqnarray*}
\overline{S}_{\ell}&=&\left\{ (r,\omega): r \in \mathbb{R}_+, \,
0<\omega<\omega_{\ell} \right\},
\\
\overline{S}_{\ell+1}&=&\left\{ (r,\omega): r \in \mathbb{R}_+, \,
\omega_{\ell}<\omega<\omega_{\ell+1} \right\},
\end{eqnarray*}
respectively, and $\overline{\bfw}_{m}(r,\omega)=\bfw_{m}(x_1,x_2)$,
$\overline{\bfF}_{m}(r,\omega)=\bfF_{m}(x_1,x_2)$ and
$\overline{\bfG}_{m}(r,\omega)=\bfG_{m}(x_1,x_2)$, $m=\ell,\ell+1$.
Substituting $r=e^\xi$, we get the system of equations
\begin{equation}\label{transmission_problem_exp}
\left\{
\begin{array}{rcllll}
- \varepsilon^{ji}_{m}({\bf0})\left( \frac{\partial ^2
\widetilde{w}^i_{m}}{\partial \xi^2} + \frac {\partial ^2
\widetilde{w}^i_{m}}{\partial \omega ^2} \right) &=&
\widetilde{F}^j_{m} \; {\rm in }\; \widetilde{S}_{m}, \;
m=\ell,\ell+1, & &&
\\
\varepsilon^{ji}_{\ell}({\bf0})\frac{\partial
\widetilde{w}^i_{\ell}}{\partial \omega}(\xi,0)&=&
\widetilde{G}^j_{\ell}(\xi,0), &  &  &
\\
\varepsilon^{ji}_{\ell+1}({\bf0})\frac{\partial
\widetilde{w}^i_{\ell+1}}{\partial \omega}(\xi,\omega_{\ell+1})&=&
\widetilde{G}^j_{\ell+1}(\xi,\omega_{\ell+1}), &  & &
\\
\widetilde{w}^j_{\ell}(\xi,\omega_{\ell})  &=&
\widetilde{w}^j_{\ell+1}(\xi,\omega_{\ell}),  &    &
\\
\varepsilon^{ji}_{\ell}({\bf0})\frac{\partial
\widetilde{w}^i_{\ell}}{\partial
\omega}(\xi,\omega_{\ell})&=&\varepsilon^{ji}_{\ell+1}({\bf0})\frac{\partial
\widetilde{w}^i_{\ell+1}}{\partial \omega}(\xi,\omega_{\ell}),&&
\end{array}\right.
\end{equation}
where $\widetilde{S}_{\ell}$ and $\widetilde{S}_{\ell+1}$,
respectively, denotes an infinite strip
\begin{eqnarray*}
\widetilde{S}_{\ell}&=&\left\{ (\xi,\omega): \xi \in \mathbb{R}, \,
0<\omega<\omega_{\ell} \right\},\\
\widetilde{S}_{\ell+1}&=&\left\{ (\xi,\omega): \xi \in \mathbb{R},
\, \omega_{\ell}<\omega<\omega_{\ell+1} \right\},
\end{eqnarray*}
respectively, and $\widetilde{\bfw}(\xi,\omega)=\bfw(x_1,x_2)$,
$\widetilde{\bfF}_m(\xi,\omega)e^{-2\xi}=\overline{\bfF}_{m}(r,\omega)$,
$\widetilde{\bfG}_m(\xi,\omega)e^{-\xi}=\overline{\bfG}_m(r,\omega)$.
\begin{figure}[h]
\centering
\includegraphics[width=5.3cm]{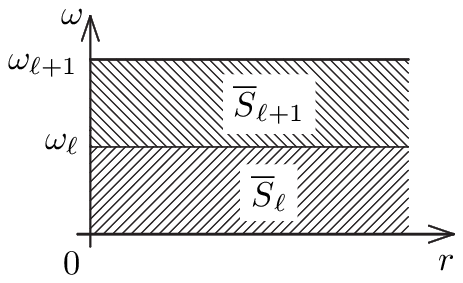}
%\caption{The admissible domain $\Omega$.}\label{composite}
\qquad \includegraphics[width=6.1cm]{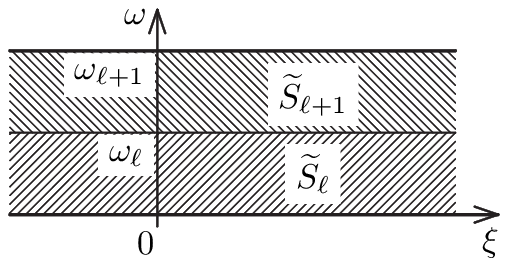} \caption{The
infinite half--strip
$\overline{S}=\overline{S}_{\ell}\cup\overline{S}_{\ell+1}$ and the
infinite strip
$\widetilde{S}=\widetilde{S}_{\ell}\cup\widetilde{S}_{\ell+1}$.}
\end{figure}
We apply the complex Fourier transform $\mathcal{F}_{\xi\rightarrow
\lambda}$ with respect to real variable $\xi \in \mathbb{R}$,
\begin{displaymath}
[\mathcal{F}_{\xi\rightarrow \lambda}\phi(\xi)](\lambda) =
\widehat{\phi}(\lambda)=\frac{1}{\sqrt{2\pi}}
\int\limits_{-\infty}^{\infty} \phi(\xi)e^{-\mathrm{i} \lambda \xi}
\;{\rm d}\xi, \quad \lambda \in \mathbb{C},
\end{displaymath}
to transform \eqref{transmission_problem_exp} to the one-dimensional
problem
\begin{equation}\label{transmission_problem_fourier}
\left\{
\begin{array}{rclll}
- \varepsilon^{ji}_{\ell}({\bf0})\left(
(\mathrm{i}\lambda)^2\widehat{w}^i_{\ell} + \frac {\partial ^2
\widehat{w}^i_{\ell}}{\partial \omega ^2} \right) &=&
\widehat{F}^j_{\ell} \textmd{ for } \omega\in(0,\omega_{\ell}),
\\
- \varepsilon^{ji}_{\ell+1}({\bf0})\left(
(\mathrm{i}\lambda)^2\widehat{w}^i_{\ell+1} + \frac {\partial ^2
\widehat{w}^i_{\ell+1}}{\partial \omega ^2} \right) &=&
\widehat{F}^j_{\ell+1} \textmd{ for }
\omega\in(\omega_{\ell},\omega_{\ell+1}),
\\
\varepsilon^{ji}_{\ell}({\bf0})\frac{\partial
\widehat{w}^i_{\ell}}{\partial \omega}(\lambda,0)&=&
\widehat{G}^j_{\ell}(\lambda,0),
\\
\varepsilon^{ji}_{\ell+1}({\bf0})\frac{\partial
\widehat{w}^i_{\ell+1}}{\partial \omega}(\lambda,\omega_{\ell+1})&=&
\widehat{G}^j_{\ell+1}(\lambda,\omega_{\ell+1}),
\\
\widehat{w}^j_{\ell}(\lambda,\omega_{\ell})  &=&
\widehat{w}^j_{\ell+1} (\lambda,\omega_{\ell}),
\\
\varepsilon^{ji}_{\ell}({\bf0})\frac{\partial
\widehat{w}^i_{\ell}}{\partial \omega}(\lambda,\omega_{\ell}) &=&
\varepsilon^{ji}_{\ell+1}({\bf0})\frac{\partial
\widehat{w}^i_{\ell+1}}{\partial\omega}(\lambda,\omega_{\ell})
\end{array}\right.
\end{equation}
with complex parameter $\lambda$. Solvability of the parameter
dependent boundary value problems were studied in \cite{KufSan1987}.
Roughly speaking, the operator pencil
$\widehat{\mathfrak{A}}_A(\lambda)$, associated with the parameter
dependent boundary value problem
\eqref{transmission_problem_fourier}, is an isomorphism for all
complex parameters $\lambda \in \mathbb{C}$ except at certain
isolated points -- the eigenvalues of
$\widehat{\mathfrak{A}}_A(\lambda)$ (for precise definition of
eigenvalues and corresponding eigensolutions we refer to monograph
\cite{KufSan1987}). As well-known, the regularity of the weak
solution $\bfu_{\ell}\in \mathbf{W}^{1,2}_{\ell}$ (the existence of
the strong solution, i.e. whether or not $\bfu_{\ell}\in
\mathbf{W}^{2,2}_{\ell}$) depends on the distribution of the
eigenvalues $\lambda$  of the operator pencil
$\widehat{\mathfrak{A}}_A(\lambda)$ in the strip ${\rm
Im}\,\lambda\in(-1,0)$. This assertion is expressed by the following
theorem, which is a classical result, see \cite{Dauge1988},
\cite{Kon1967}, \cite{KozMazRoss1997}, \cite{KozMazRoss2001},
\cite{KufSan1987}:
\begin{thm}[Regularity theorem in an infinite angle]\label{theorem pom 2}
Let  $\bfw_{\ell}\in\mathbf{W}^{1,2}_{\ell}(\mathcal{K}_{\ell})$  be
the uniquely determined weak solution of
\eqref{transmission_problem_angle}, $\bfF_{\ell} \in
\mathbf{L}_{\ell}^2(\mathcal{K}_{\ell})$,  $\bfG_{\ell} \in
\mathbf{W}_{\ell,\Gamma}^{1/2,2}(\mathcal{K}_{\ell})$. If the strip
${\rm Im}\,\lambda\in(-1,0)$ is free of eigenvalues of the operator
pencil $\widehat{\mathfrak{A}}_A(\lambda)$, then
$\bfw_{\ell}\in\mathbf{W}^{2,2}_{\ell}(\mathcal{K}_{\ell})$ and
\begin{equation*}
\|\bfw_{\ell}\|_{\mathbf{W}^{2,2}_{\ell}(\mathcal{K}_{\ell})} \leq c
 \left( \|\bfF_{\ell}\|_{\mathbf{L}_{\ell}^2(\mathcal{K}_{\ell})}+
 \|\bfG_{\ell}\|_{\mathbf{W}_{\ell,\Gamma}^{1/2,2}(\mathcal{K}_{\ell})}
 \right).
\end{equation*}
\end{thm}
\begin{proof}
Theorem \ref{theorem pom 2} is a consequence of \cite[\S 7, Theorem
7.5]{KufSan1987} and the expansion of the solution \cite[\S 7,
(7.10)]{KufSan1987}.
\end{proof}

%%%%%%%%%%%%%%%%%%%%%%%%%%%%%%%%%%%%%%%%%%%%%%%%%%%%%%%%%%%%%%%%%%%%%%%%%%%%%%%%

\subsubsection*{The characteristic determinants and the distribution of the
eigenvalues of $\widehat{\mathfrak{A}}_A(\lambda)$ }
Every $\lambda_0 \in \mathbb{C}$ such that ker
$\widehat{\mathfrak{A}}_A(\lambda_0)\neq \left\{ \bf0 \right\}$ is
said to be an eigenvalue of $\widehat{\mathfrak{A}}_A(\lambda)$. The
distribution of the eigenvalues of the operator
$\widehat{\mathfrak{A}}_A(\lambda)$ plays crucial role in the
regularity results of the solution, see Theorem \ref{theorem pom 2}.
In order to calculate the eigenvalues of the operator pencil
$\widehat{\mathfrak{A}}_A(\lambda)$ we look those $\lambda$, for
which there exists the nontrivial solution of the system
\eqref{transmission_problem_fourier} with the vanishing right-hand
side. The general solution $[\widehat{e}^1_{m},\widehat{e}^2_{m}]$
of the homogeneous equations
\begin{eqnarray*}
- \varepsilon^{11}_{m}({\bf0})\left(
(\mathrm{i}\lambda)^2\widehat{e}^1_{m} + \frac {\partial ^2
\widehat{e}^1_{m}}{\partial \omega ^2} \right) -
\varepsilon^{12}_{m}({\bf0})\left(
(\mathrm{i}\lambda)^2\widehat{e}^2_{m} + \frac {\partial ^2
\widehat{e}^2_{m}}{\partial \omega ^2} \right) &=& 0,
\\
- \varepsilon^{21}_{m}({\bf0})\left(
(\mathrm{i}\lambda)^2\widehat{e}^1_{m} + \frac {\partial ^2
\widehat{e}^1_{m}}{\partial \omega ^2} \right) -
\varepsilon^{22}_{m}({\bf0})\left(
(\mathrm{i}\lambda)^2\widehat{e}^2_{m} + \frac {\partial ^2
\widehat{e}^2_{m}}{\partial \omega ^2} \right) &=& 0,
\end{eqnarray*}
$m=\ell,\ell+1$, has the form (recall that $\varepsilon^{ji}_{m}$ is
a positive definite matrix)
\begin{equation}\label{general_solution}
\left\{
\begin{array}{lcl}
\widehat{e}^1_{\ell}&=& C_1 \cos(\mathrm{i}\lambda\omega)
+C_2\sin(\mathrm{i}\lambda\omega) \textmd{ for }
\omega\in(0,\omega_{\ell}),
\\
\widehat{e}^2_{\ell}&=& C_3 \cos(\mathrm{i}\lambda\omega)
+C_4\sin(\mathrm{i}\lambda\omega) \textmd{ for }
\omega\in(0,\omega_{\ell}),
\\
\widehat{e}^1_{\ell+1}&=& C_5 \cos(\mathrm{i}\lambda\omega)
+C_6\sin(\mathrm{i}\lambda\omega) \textmd{ for }
\omega\in(\omega_{\ell},\omega_{\ell+1}),
\\
\widehat{e}^2_{\ell+1}&=&C_7\cos(\mathrm{i}\lambda\omega)
+C_8\sin(\mathrm{i}\lambda\omega)\textmd{ for }
\omega\in(\omega_{\ell},\omega_{\ell+1}).
\end{array}
\right.
\end{equation}

The eigenvalues of $\widehat{\mathfrak{A}}_A(\lambda)$ are zeros of
the determinant $D_A(\lambda)$ of corresponding matrix of
coefficients $C_1,\dots,C_8$ (substituting the general solution
(\ref{general_solution}) to the corresponding boundary conditions
and transmission conditions, respectively, we get the homogeneous
linear system of eight equations with unknowns $C_1,\dots,C_8$).
Computation of $D_A(\lambda)$ %\eqref{det}
leads to the transcendent equation
\begin{equation}\label{determinant}
D_A(\lambda) =
D^{11}_A(\lambda)D^{22}_A(\lambda)-D^{12}_A(\lambda)D^{21}_A(\lambda)
=0,
\end{equation}
where
\begin{eqnarray*}
D^{11}_A(\lambda)&=&\varepsilon^{11}_{\ell}({\bf0})\sin(\mathrm{i}\lambda\omega_{\ell})
                    \cos[\mathrm{i}\lambda(\omega_{\ell+1}-\omega_{\ell})]
                    \\
                  &&+\varepsilon^{11}_{\ell+1}({\bf0})\cos(\mathrm{i}\lambda\omega_{\ell})
                  \sin[\mathrm{i}\lambda(\omega_{\ell+1}-\omega_{\ell})],
                    \\
D^{12}_A(\lambda)&=&\varepsilon^{12}_{\ell}({\bf0})\sin(\mathrm{i}\lambda\omega_{\ell})
                    \cos[\mathrm{i}\lambda(\omega_{\ell+1}-\omega_{\ell})]
                    \\
                  &&+\varepsilon^{12}_{\ell+1}({\bf0})\cos(\mathrm{i}\lambda\omega_{\ell})
                  \sin[\mathrm{i}\lambda(\omega_{\ell+1}-\omega_{\ell})],
                      \\
D^{21}_A(\lambda)&=&\varepsilon^{21}_{\ell}({\bf0})\sin(\mathrm{i}\lambda\omega_{\ell})
                    \cos[\mathrm{i}\lambda(\omega_{\ell+1}-\omega_{\ell})]
                    \\
                  &&+\varepsilon^{21}_{\ell+1}({\bf0})\cos(\mathrm{i}\lambda\omega_{\ell})
                  \sin[\mathrm{i}\lambda(\omega_{\ell+1}-\omega_{\ell})],
                      \\
D^{22}_A(\lambda)&=&\varepsilon^{22}_{\ell}({\bf0})\sin(\mathrm{i}\lambda\omega_{\ell})
                    \cos[\mathrm{i}\lambda(\omega_{\ell+1}-\omega_{\ell})]
                    \\
                  &&+\varepsilon^{22}_{\ell+1}({\bf0})\cos(\mathrm{i}\lambda\omega_{\ell})
                  \sin[\mathrm{i}\lambda(\omega_{\ell+1}-\omega_{\ell})].
\end{eqnarray*}
The roots of the equation $D_A(\lambda)=0$  are the eigenvalues of
$\widehat{\mathfrak{A}}_A(\lambda)$.

Taking into account the special type of geometry, namely
$\omega_{\ell}=\omega_{\ell+1}/2$, \eqref{determinant} simplifies into
\begin{multline}
\frac{1}{2}\left[(\varepsilon^{11}_{\ell}({\bf0})+\varepsilon^{11}_{\ell+1}({\bf0}))
(\varepsilon^{22}_{\ell}({\bf0})+\varepsilon^{22}_{\ell+1}({\bf0}))
\right.
\\
\left.
-(\varepsilon^{12}_{\ell}({\bf0})+\varepsilon^{12}_{\ell+1}({\bf0}))
(\varepsilon^{21}_{\ell}({\bf0})+\varepsilon^{21}_{\ell+1}({\bf0}))\right]
\sin(2\mathrm{i}\lambda\omega_{\ell})=0.
\end{multline}
Since both matrices, $\varepsilon^{ij}_{\ell}$ and
$\varepsilon^{ij}_{\ell+1}$, are considered to be positive definite,
we get
\begin{equation}
\sin(2\mathrm{i}\lambda\omega_{\ell})=0,
\end{equation}
from whence we obtain
$$
\mathrm{i}\lambda = \frac{k\pi}{2\omega_{\ell}}, \quad k\in
\mathbb{Z}.
$$
Now it is clear that for $\omega_{\ell}\in(0,\pi/2]$ there are no
roots of the equation $D_A(\lambda)=0$ such that ${\rm
Im}\,\lambda\in(-1,0)$.
\begin{crl}\label{corollary_appendix}
Let  $\bfu_{\ell} \in \mathbf{W}^{1,2}_{\ell}$  be the uniquely
determined weak solution of \eqref{transm_elliptic_system},
$\bff_{\ell} \in \mathbf{L}_{\ell}^2$,  $\bfg_{\ell} \in
\mathbf{W}_{\ell,\Gamma}^{1/2,2}$. Since the strip ${\rm
Im}\,\lambda\in(-1,0)$ is free of eigenvalues of the operator
$\widehat{\mathfrak{A}}_A(\lambda)$, we have
$\bfu_{\ell}\in\mathbf{W}^{2,2}_{\ell}$ and
\begin{equation*}
\|\bfu_{\ell}\|_{\mathbf{W}^{2,2}_{\ell}} \leq c
 \left( \|\bff_{\ell}\|_{\mathbf{L}_{\ell}^2}
 + \|\bfg_{\ell}\|_{\mathbf{W}_{\ell,\Gamma}^{1/2,2}} \right).
\end{equation*}
\end{crl}
\begin{proof}[Sketch of the proof]
The assertion follows from Theorem \ref{theorem pom 2} and the
determinant equation \eqref{determinant}. \eqref{determinant}
implies that for $\omega_{\ell+1}\leq\pi$,
$\omega_{\ell}=\omega_{\ell+1}/2$ there are no eigenvalues of the
operator $\widehat{\mathfrak{A}}_A(\lambda)$ in the strip ${\rm
Im}\,\lambda\in(-1,0)$. Hence
$\bfw_{\ell}\in\mathbf{W}^{2,2}_{\ell}(\mathcal{K}_{\ell})$. Now
consider the weak solution $\bfu_{\ell} \in \mathbf{W}^{1,2}_{\ell}$
of \eqref{transm_elliptic_system}.  Let $\mathcal{M}_{\ell}$ be the
set of all boundary corner points $A\in\partial\Omega_{\ell}$,
$\ell=1,\dots,M$. We have
\begin{eqnarray}
\bfu_{\ell} &=& \left(
1-\sum_{A\in\mathcal{M}_{\ell}}\eta_A\right)\bfu_{\ell}  +
\sum_{A\in\mathcal{M}_{\ell}}\eta_A\bfu_{\ell}
\\
&=&  \left( 1-\sum_{A\in\mathcal{M}_{\ell}}\eta_A\right)\bfu_{\ell}
+\sum_{A\in\mathcal{M}_{\ell}}\bfw_{\ell}.
\end{eqnarray}
The regularity of the first term on the right-hand side follows from
the interior regularity. The smoothness of the second term follows
from the regularity result in an infinite angle
$\mathcal{K}_{\ell}$, Theorem \ref{theorem pom 2}.
\end{proof}

%--------------------------------------------------------------------------------



%% References

\subsection*{Acknowledgment}
This outcome has been achieved with the financial support of the
Ministry of Education, Youth and Sports of the Czech Republic,
project No. 1M0579, within activities of the CIDEAS research centre.
In addition, research of the first author was partly covered by the
grant 201/09/1544 (Czech Science Foundation). The work of the second
author was suppored by the Czech Science Foundation through projects
No.~103/08/1531 and No.~201/10/0357.
% ------------------------------------------------------------------------

%--------------------------------------------------------------------------------

% ------------------------------------------------------------------------


\bibliographystyle{elsarticle-num}






% \end{thebibliography}

\end{document}